\documentclass[bj]{imsart}

\RequirePackage{amsthm,amsmath,amssymb}
\RequirePackage[numbers]{natbib}

\startlocaldefs
\numberwithin{equation}{section}
\theoremstyle{plain}
\newtheorem{thm}{Theorem}[section]

\newtheorem{lemma}{Lemma}[section]
\newtheorem{remark}{Remark}[section]

\endlocaldefs

\begin{document}
\global\long\def\xib{\boldsymbol{\xi}}
\global\long\def\etab{\boldsymbol{\eta}}
\global\long\def\thetab{\boldsymbol{\theta}}
\global\long\def\yb{\boldsymbol{y}}
\global\long\def\xb{\boldsymbol{x}}
\global\long\def\eb{\boldsymbol{e}}
\global\long\def\0b{\boldsymbol{0}}
\global\long\def\zb{\boldsymbol{z}}
\global\long\def\zetab{\boldsymbol{\zeta}}
\global\long\def\taub{\boldsymbol{\tau}}
\global\long\def\wb{\boldsymbol{w}}
\begin{frontmatter}
\title{ On uniform consistency \\ of nonparametric tests I}
\footnote{The research has been supported by RFFI Grant  20-01-00273.}
\runtitle{ On uniform consistency }

\begin{aug}
\author{\fnms{Mikhail} \snm{Ermakov}\ead[label=e1]{erm2512@gmail.com}}

\runauthor{M. Ermakov}

\address{Institute of Problems of Mechanical Engineering RAS, Bolshoy pr., 61, VO, 1991178  St. Petersburg and
St. Petersburg State University, Universitetsky pr., 28, Petrodvoretz, 198504 St. Petersburg, RUSSIA\\
\printead{e1}}

\end{aug}

\begin{abstract}For widespread nonparametric tests we point out necessary and sufficient conditions of uniform consistency of nonparametric sets of alternatives approaching to hypothesis. Nonparametric sets of alternatives can be defined both in terms of distribution functions and in terms of densities (or signals in the problem of signal detection in Gaussian white noise). In this part of paper such conditions are provided
for   $\chi^2-$tests having increasing number of cells with growth of sample size, Cramer-von Mises tests, tests generated  $\mathbb{L}_2$- norms of kernel estimators and  tests generated quadratic forms of estimators of Fourier coefficients.
  \end{abstract}



\end{frontmatter}

\maketitle

\section{Introduction}

Let  $X_1,\,\ldots,\,X_n$ be sample of i.i.d.r.v.'s having c.d.f.  $F \in \Im$. Here $\Im $ is set of all distribution functions of random variables having values into interval (0,1)

We explore problem of  testing  hypothesis
\begin{equation} \label{ax1}
\mathbb{H}_0\,:\, F(x) = F_0(x)= x, \quad x \in [0,1]
\end{equation}
versus sets of alternatives defined in terms of
\vskip 0.25cm
distribution functions
\begin{equation}\label{ax2}
\mathbb{H}_n \,:\, F \in \Upsilon_n, \quad \Upsilon_n \subset \Im
\end{equation}
or in terms of densities $p(x) = 1 + f(x)  =  \frac{dF(x)}{dx}$
\begin{equation}\label{axa}
\mathbb{H}_{1n} \,:\, f \in \Psi_n, \quad \Psi_n \subset \mathbb{L}_2(0,1).
\end{equation}
Here $\mathbb{L}_2(0,1)$ is Hilbert space of all quadratically integrable functions $g(t)$, $t \in (0,1)$ with $\mathbb{L}_2$-norm $\|g\| = \left(\int g^2(t)\,d\,t\,\right)^{1/2}$.

For part of setups the problem of goodness of fit testing for distribution function or density is replaced with the problem of signal detection in Gaussian white noise. This allows to simplify technical part of paper.

We are interested in uniform consistency of nonparametric tests. If test or test statistic is uniformly consistent for sets of alternatives, we say that these sets of alternatives are uniformly consistent for these tests or test statistics.

For setups mentioned above we  point out necessary and sufficient conditions of uniform consistency of sets of alternatives (\ref{ax2}) and  (\ref{axa}) for test statistics of
\vskip 0.1cm
Kolmogorov tests;
 \vskip 0.1cm
Cramer-von Mises tests;
 \vskip 0.1cm
 chi-squared tests having increasing number of cells with growth of sample size;
 \vskip 0.1cm
 tests generated quadratic forms of estimators of Fourier coefficients of orthogonal expansion of signal;
 \vskip 0.1cm
 tests generated $\mathbb{L}_2$ --norms of kernel estimators.
\vskip 0.1cm
Last four of above mentioned tests statistics have  quadratic structure. The results and proofs for these test statistics are similar. We provide these results in first part of paper. The results about Kolmogorov tests are provided in second part of paper.

Denote $\hat F_n$ -- empirical distribution function of $X_1,\ldots,X_n$.

If sets of alternatives are defined in terms of distribution functions, necessary and sufficient conditions of consistency will be provided in the framework of distance  method.

 Test statistics can be considered as functionals $T_n(\hat F_n)$ depending on empirical distribution functions. Functionals
  $T_n(F)$  admits interpretation as norms or seminorms defined on the set of differences of distribution functions. Established uniform consistency of tests statistics on sets of alternatives
$$
\Upsilon_n(T_n,\rho_n) = \{\, F \,: \, T_n(F) > \rho_n > 0, \, F \in \Im \,\}
$$
allows to make a conclusion about uniform consistency of any sequence of sets of alternatives
 $\Upsilon_n$ in terms of their distances  or semidistances $$
\inf_{F \in \Upsilon_n} T_n(F)$$ from hypothesis.

For specially selected sequences  $\rho_n$, $\rho_n \to 0$  as $n \to \infty$, in papers \cite{er97, er04, er03} (see Theorems \ref{chi2}, \ref{tq2}, \ref{tk2} as well) we established uniform consistency of sets $\Upsilon_n(T_n,\rho_n)$ of alternatives for $\chi^2-$tests having increasing number of cells with growth of sample size, tests generated  $\mathbb{L}_2$- norms of kernel estimators and tests generated quadratic forms of estimators of Fourier coefficients
Moreover asymptotic minimaxity of tests on these sets has been established.  In this part of paper we establish uniform consistency of sets  $\Upsilon_n(T,\rho_n)$ of alternatives for Cramer - von Mises test (see  Theorem \ref{tcm}.
Some similar results will be established  for Kolmogorov test in the second part of paper.

Proof of results on uniform consistency of sets of alternatives (\ref{axa}) defined in terms of densities or signals are based on these results.

Problem of signal detection is considered for the following setup. We observe a realization of random process
 $Y_n(t)$ defined stochastic differential equation
\begin{equation}\label{q1}
d\,Y_n(t) = f(t)\, dt + \frac{\sigma}{\sqrt{n}}\, d\,w(t), \quad t \in [0,1],\quad \sigma >0,
\end{equation}
where  $f \in \mathbb{L}_2(0,1)$ is unknown signal and  $dw(t)$ is Gaussian white noise.

  The following nonparametric sets of alternatives (see \cite{dal, er90,er04, ing87,ing02,ing12,la,lep}) are often explored
\begin{equation}\label{i26}
\mathbb{H}_n\, : \, f \in V_n = \{\, f \,: \, \|f\|^2 \ge \rho_n,\, f \in U \subset \mathbb{L}_2(0,1)\, \},
\end{equation}
where $\rho_n \to 0$ as $n \to \infty$. Here  $U$  is a convex set.

We answer on  four questions given bellow. The answer on the first question is provided for problem of signal detection in Gaussian white boise and  does not touch test statistics mentioned above.
\vskip 0.2cm
{ \sl For which bounded convex sets $U$ there are $\rho_n \to 0$ as $n \to \infty$ such that there is uniformly consistent sequence of tests for sets $V_n$ of alternatives ? }
\vskip 0.2cm
We show that uniformly consistent test exists, if and only if,  set $U$ is relatively compact (see Theorems \ref{tqq} and \ref{tqq1}). Note that  necessary and sufficient condition of existence of consistent nonparametric estimator on nonparametric set  is   relative compactness of this set \cite{ih}, \cite{jo}. The same compactness condition arises in solution of ill-posed inverse problems with deterministic errors  \cite{en}. The problem of existence of consistent tests has been explored for different setups. The most complete bibliography one can find in \cite{er15}.

The answer on the next three questions is provided for i.i.d.r.v.'s model in the case of Cramer-von Mises tests and chi-squared tests. Test statistics generated quadratic forms of estimators of Fourier coefficients or tests generated  $\mathbb{L}_2$- norms of kernel estimators are explored for problem of signal detection in Gaussian white noise.
\vskip 0.2cm
{ \sl Let $\rho_n = n^{-r}$, $0 < r \le 1/2$, and  $r$ is fixed. How to define  biggest bounded sets $U$  such that sets $V_n$ are uniformly consistent for one of above mentioned test statistics  ?}
\vskip 0.2cm
We call such sets $U$-- maxisets The exact definition of maxisets is provided in section
 \ref{sec2}. For $0< r <1/2$, for test statistics having quadratic structure we show  (see Theorems \ref{tq1}, \ref{tk3}, \ref{tchi3}, \ref{tom1}), that maxisets are bodies in Besov spaces $\mathbb{B}^s_{2\infty}(P_0)$, $P_0  > 0$. Here $ r = \frac{2s}{1 + 4s}$ for chi-squared test statistics, test statistics being  $\mathbb{L}_2$- norms of kernel estimator and test statistics being quadratic forms of estimators of Fourier coefficients of signal. For Cramer- von Mises tests we have  $r = \frac{s}{2+2s}$.

 If $ r = 1/2 $, we could not find sets satisfying all  requirements of the definition of maxisets. However, we show that bounded convex sets of functions having a fixed finite  number of nonzero Fourier coefficients  satisfy similar requirements.  In further statements of this section for $ r = 1/2 $, and therefore in the corresponding theorems, the maxisets can be replaced with such sets.

Uniform consistency of chi-squared tests and Cramer-von Mises tests for above mentioned Besov bodies has been established Ingster \cite{ing87}.

For  nonparametric estimation the notion of maxisets has been introduced Kerkyacharian and Picard \cite{ker93}. Maxisets of  nonparametric  estimators have been comprehensively explored in \cite{co},   \cite{ker02},  \cite{rio} (see also references therein). For nonparametric hypothesis testing completely different definition of maxisets has been introduced  Autin, Clausel,  Freyermuth and  Marteau \cite{au}.

 Let each set  $\Psi_n$ be bounded in $\mathbb{L}_2(0,1)$. Then Cramer- von Mises tests, chi-squared tests, tests generated $\mathbb{L}_2$-norms of kernel estimators and quadratic forms of estimators of Fourier coefficients of signal are uniformly consistent, if and only if,  these sets $\Psi_n$ of alternatives does not contain inconsistent sequence of simple alternatives $f_n \in \Psi_n$. In other words sets of alternatives are uniformly consistent, if and only if, all sequences of simple alternatives $f_n \in \Psi_n$ are consistent. Thus the problem of uniform consistency for sets  $\Psi_n$ of alternatives is reduced to the problem of consistency of any sequence of simple alternatives   $f_n \in \Psi_n$.
\vskip 0.2cm
{\sl How to describe all consistent and inconsistent sequences of simple alternatives having given rate of convergence to hypothesis ?}
\vskip 0.2cm
   We explore this problem as problem of testing hypothesis
   \begin{equation}\label{i29}
\mathbb{\bar H}_0\,:\, f(x)  = 0, \quad x\in [0,1],
\end{equation}
    versus sequence of simple alternatives
 \begin{equation}\label{i30}
\mathbb{\bar H}_n\,:\, f = f_n, \qquad c\,n^{-r} \le \| f_n\| \le C\, n^{-r},
\end{equation}
where $0 < r \le 1/2$ and $0 < c < C <\infty$.

For above mentioned test statistics  answer on this question is provided in terms of concentration of Fourier coefficients (Theorems \ref{tq3} and \ref{tq4}). In Theorem  \ref{tq7}  we propose the following interpretation  of  these results:
\vskip 0.2cm
  {\it sequence of simple alternatives  $f_n$, $c\,n^{-r} \le \| f_n\| \le C\, n^{-r}$, is consistent, if and only if, functions $f_n$    admit representation as
functions $f_{1n}$ from maxiset with the same rate of convergence to hypothesis plus functions $f_{n} - f_{1n}$ orthogonal to functions $f_{1n}$.}
\vskip 0.2cm
  In Theorem   \ref{tq11} we show that, for any $\varepsilon > 0$, there are maxiset and functions $f_{1n}$  from maxiset such that the differences of type II error probabilities for alternatives $f_{n}$ and $f_{1n}$ is smaller   $\varepsilon $ and  $f_{1n}$ is orthogonal to $f_{n} - f_{1n}$.

    Thus, each  function of consistent sequence of alternatives  with fixed rate of convergence to hypothesis  contains sufficiently smooth function  as an additive component and this function carries almost all information on its type II error probability.
\vskip 0.2cm
  {\sl What can we say about properties of consistent and inconsistent sequences of alternatives having fixed rate of convergence to hypothesis in $\mathbb{L}_2$- norm?}
\vskip 0.2cm
 In Theorem \ref{tq5} we establish that asymptotic of type II error probabilities of sums of alternatives from consistent and inconsistent sequences coincides with the asymptotic for consistent sequence.

We call sequence of alternatives $f_n$ purely consistent if there does not exist inconsistent sequence of alternatives $f_{2n}$ having the same rates of convergence to hypothesis and such that $f_{2n}$ are orthogonal to $f_n  - f_{2n}$.

It is easy to  show that any sequence of alternatives from maxisets with fixed rates of convergence to hypothesis is purely consistent.

In Theorem \ref{tq6}, in terms of concentration of Fourier coefficients we point out  analytic assignment of purely consistent sequences of alternatives.

 In Theorem \ref{tq12} we show that, for any $\varepsilon>0$, for any purely consistent sequence of alternatives $f_n $, $cn^{-r} \le \| f_n\| \le C n^{-r}$,  there are maxiset and some sequence  $f_{1n}$ from this maxiset, such that there holds $\|f_n - f_{1n}\| \le \varepsilon n^{-r}$.
\vskip 0.15cm
 Paper is organized as follows. In section \ref{sec2} we introduce main definitions. In section \ref{sec3},  the answer on the first question is provided. In sections  \ref{sec4}, \ref{sec5}, \ref{sec6} and \ref{sec7}, for $0 < r <1/2$,  above mentioned  results are established respectively for  test statistics based on quadratic forms of estimators of Fourier coefficients,  $\mathbb{L}_2$ -- norms of kernel estimators, $\chi^2$--tests and Cramer-- von Mises tests.
   In section \ref{sec8} we focus on the case $r = \frac{1}{2}$.

    Proof of all Theorems is provided in Appendix.

We use letters $c$ and $C$ as a generic notation for positive constants. Denote ${\bf 1}_{\{A\}}$ the
indicator of an event $A$.  Denote $[a]$ whole part of real number $a$. For any two sequences of positive real numbers $a_n$ and $b_n$,  $a_n \asymp b_n$ implies $c < a_n/b_n < C$ for all $n$ and $a_n = o(b_n)$ implies $a_n/b_n \to 0$ as $n \to \infty$. For any complex number $z$ denote $\bar z$ complex conjugate number.

Denote
  $$
  \Phi(x) = \frac{1}{\sqrt{2\pi}}\,\int_{-\infty}^x\,\exp\{-t^2/2\}\, d\,t, \quad x \in \mathbb{R}^1,
$$
standard normal distribution function.

Let $\phi_j$, $1 \le j < \infty$, be  orthonormal system of functions in $\mathbb{L}_2(0,1)$. For each  $P_0 > 0$ define set
\begin{equation}\label{vv}
\mathbb{\bar B}^s_{2\infty}(P_0) = \Bigl\{\,f : f = \sum_{j=1}^\infty\,\theta_j\phi_j,\,\,\,  \sup_{\lambda>0} \lambda^{2s}\, \sum_{j>\lambda}\, \theta_j^2 \le P_0,\,\, \theta_j \in \mathbb{R}^1\, \Bigr\}.
\end{equation}
If some assumptions about basis $\phi_j$, $1 \le j < \infty,$  holds, functional space
$$
\bar{\mathbb{ B}}^s_{2\infty} = \Bigl\{\, f : f = \sum_{j=1}^\infty\,\theta_j\phi_j,\,\,\,  \sup_{\lambda>0} \lambda^{2s}\, \sum_{j>\lambda}\, \theta_j^2 < \infty,\,\, \theta_j \in \mathbb{R}^1\, \Bigr\}
$$
is Besov space $\mathbb{B}^s_{2\infty}$ (see   \cite{rio}).
In particular, $\mathbb{\bar B}^s_{2\infty}$ is Besov space if $\phi_j$, $1 \le j < \infty$, is trygonometric basis.

If $\phi_j(t) = \exp\{2\pi i j x\}$, $x\in (0,1)$, $j = 0, \pm 1, \ldots$,  denote
$$
\mathbb{ B}^s_{2\infty}(P_0) = \Bigl\{\,f : f = \sum_{j=-\infty}^\infty\\, \theta_j\,\phi_j,\,\,\,  \sup_{\lambda>0}\, \lambda^{2s}\, \sum_{|j| >\lambda} |\theta_j|^2 \le P_0\, \Bigr\}.
$$
Since here  $\phi_j$ are complex functions, then $\theta_j$ are complex numbers as well and $\theta_j =  \bar\theta_{-j}$ for all $-\infty < j < \infty$.

For the same basis denote
$$
\mathbb{\tilde B}^s_{2\infty}(P_0) = \Bigl\{\,f : f = \sum_{j=-\infty}^\infty\, \theta_j\,\phi_j,\,\,f \in \mathbb{ B}^s_{2\infty}(P_0),\,  \theta_0 =0\, \Bigr\}.
$$
Balls in Nikolskii classes
$$
\int\,(f^{(l)}(x+t) - f^{(l)}(x))^2\, d\,x \le L |t|^{2(s-l)}, \quad \|f\| < C,
$$
are Besov balls in  $\mathbb{B}^s_{2\infty}$. Here $l = [s]$.
\section{Main definitions \label{sec2}}
\subsection{ Consistency and $n^{-r}$-consistency \label{ss2.1}}
For any test $K_n$  denote $\alpha(K_n)$ its type I error probability, and $\beta(K_n,f)$ its type II error probability for  alternative $f \in \mathbb{L}_2(0,1)$. Similar  notation  $\beta(K_n,F)$ is implemented if alternative is c.d.f.  $F$.

Definition of consistency  will be slightly different in each  section. In section \ref{sec3}  problem of existence of uniformly consistent tests and uniform consistency of sets of alternatives is considered among all tests.

In section \ref{sec4} consistency is considered for a fixed sequence of test statistics $T_n$. For kernel-based tests and chi-squared tests, consistency is explored for whole population of test statistics depending on kernel width and number of cells respectively. In section \ref{sec7} we have only one test statistic.

We showed that problem of uniform consistency of sets of alternatives is reduced to the problem of consistency of sequences of simple alternatives. Thus, in sections \ref{sec4} - \ref{sec7}, we explore this setup.

Below we provide  definition of consistency for setup of sections \ref{sec4}  and \ref{sec7}.  In sections \ref{sec5}  and \ref{sec6} the definitions will be different in the sense mentioned above.

We say that sequence of simple alternatives $f_n$ is {\sl consistent}  if for any $\alpha$, $0 < \alpha < 1$, for sequence of tests $K_n$, $\alpha(K_n) = \alpha\,(1 + o(1))$, generated test statistics $T_n$, there holds
\begin{equation}\label{vas1}
\limsup_{n\to\infty}  \beta(K_n, f_n) < 1 - \alpha.
\end{equation}
If $cn^{-r} < \|f_n\| < Cn^{-r}$ additionally, we say that sequence of alternatives $f_n$ is $n^{-r}$- {\sl consistent} (see  \cite{ts}).

We say that sequence of alternatives $f_n$ is {\sl inconsistent}  if, for each sequence of tests $K_n$ generated test statistics $T_n$, there holds
\begin{equation}\label{vas25}
\liminf_{n\to\infty} (\alpha(K_n) + \beta(K_n, f_n)) \ge 1.
\end{equation}
Suppose we consider problem of testing hypothesis  (\ref{ax1}) versus alternatives (\ref{axa}) where $\Psi_n$ can be also sets of signals.

For tests $K_n$, $\alpha(K_n) = \alpha + o(1)$, $0 < \alpha <1$, generated test statistics $T_n$ denote
$\beta(K_n,\Psi_n) = \sup_{f \in \Psi_n}\, \beta(K_n,f)$.
We say that sequence of  sets $\Psi_n$ of alternatives is uniformly consistent if
\begin{equation}\label{uc1}
\limsup_{n \to \infty}\, \beta(K_n,\Psi_n) < 1 - \alpha.
\end{equation}
For sets of alternatives $\Upsilon_n$ defined (\ref{ax2}) definition of uniform consistency is the same.
 \subsection{Purely consistent sequences}
We say that $n^{-r}$- consistent sequence of alternatives  $f_n$ is {\sl purely $n^{-r}$-consistent} if there does not exist subsequence $f_{n_i}$ such that $f_{n_i} = f_{1n_i} + f_{2n_i}$ where  $f_{2n_i}$ is orthogonal to  $f_{1n_i}$ and sequence $f_{2n_i}$, $\|f_{2n_i}\| > c_1n^{-r}$, is inconsistent.
\subsection{Maxisets \label{ss2.3}}
Let $\phi_j$, $1 \le j < \infty$, be orthonormal basis in $\mathbb{L}_2(0,1)$. We say that a set $U$, $U \subset \mathbb{L}_2(0,1)$, is ortho-symmetric with respect to this basis if $f = \sum_{j=1}^\infty \theta_j \phi_j \in U$ implies $\tilde f = \sum_{j=1}^\infty \tilde\theta_j \phi_j \in U$ for any $\tilde\theta_j = \theta_j$ or $\tilde\theta_j = -\theta_j$, $j=1,2,\ldots$.

For closed ortho --symmetric bounded convex set $U$, $U \subset \mathbb{L}_2(0,1)$, denote $\Xi$ functional space with unite ball $U$.

For the problem of signal detection we call bounded ortho-symmetric closed  set $U$, $ U \subset \mathbb{L}_2(0,1)$,  {\sl maxiset} and functional space $\Xi$ maxispace if
\vskip 0.15cm

{\sl i.} any subsequence of alternatives $f_{n_i} \in \gamma\,U$, $cn_i^{-r} < \|f_{n_i}\| < Cn_i^{-r}$, $n_i \to \infty$ as $i \to \infty$, is consistent,
\vskip 0.25cm
{\sl ii.} if $f \notin \Xi$, then, in any convex, ortho-symmetric set $V$ that contains  $f$, there is inconsistent subsequence  of alternatives $f_{n_i} \in V$, $cn_i^{-r} < \|f_{n_i}\| < Cn_i^{-r}$, where $n_i \to \infty$ as $i \to \infty$.
\vskip 0.15cm
{\sl ii.} implies that $U$ is the largest set satisfying {\sl i.}

For problem of hypothesis testing on a density, in definition of maxiset we make additional assumption:
\vskip 0.15cm
{\sl ii.} is considered only for   functions $f = 1 + \sum_{i=1}^\infty \theta_i \phi_i$ (or $f = 1 + \sum_{|i| \ge 1}^\infty \theta_i \phi_i$) satisfying the following condition.
 \vskip 0.15cm
 {\bf D.} There is $l_0 = l_0(f)$ such that, for all $l > l_0$,  functions $1 + \sum_{|i| >l}^\infty \theta_i \phi_j$   are nonnegative (are densities).

 D allows to analyze tails $f_{n_j} = \sum_{|i| \ge j} \theta_i \phi_i$ of orthogonal expansions of $f$ to establish {\sl ii.}

 It is clear that if, $U$ is maxiset, then $\gamma U$, $0 < \gamma < \infty$, is maxiset as well.

 Simultaneous assumptions of convexity and ortho-symmetry of set $V$ is rather strong. If $f \in V$, $f = \sum_{i=1}^\infty \theta_i \phi_i$, then any $f_\eta \in V$ with  $f_\eta = \sum_{i=1}^\infty \eta_i \phi_i$,  $|\eta_i| < |\theta_i|$, $1 \le i < \infty$.

Test statistics of tests generated  $\mathbb{L}_2$- norms of kernel estimators and   Cramer-von Mises tests  admit representation as a linear combination of squares of estimators of Fourier  coefficients. Therefore, for  these test statistics, consistency of sequence $f_n $ implies consistency of any sequence of ortho-symmetric functions $\tilde f_n$ generated $f_n$. Moreover, type II error probabilities of sequences $f_n$ and $\tilde f_{n}$ have the same asymptotic. Thus  the requirement of ortho-symmetry seems natural for test statistics admitting representation as a liner combination of squares of estimators of Fourier  coefficients.
For chi-squared tests, by Theorem \ref{tchi3} given in what follows, similar situation takes place.
\subsection{Another approach to definition of maxisets \label{ss2.5}}
Requirement of ortho-symmetry of set $U$ does not allow to call maxiset any convex set $W$ generated equivalent norm in $\Xi$. In definition of maxiset given below we do not make such an assumption.

Let $\Xi \subset \mathbb{L}_2(0,1)$ be Banach space with a norm $\|\cdot\|_\Xi$. Denote $ U=\{f:\, \|f\|_\Xi \le \gamma,\, f \in \Xi\}$, $\gamma > 0,$  a ball in $\Xi$.

Define subspaces $\Pi_k$, $1 \le k < \infty$, by induction.

Denote $d_1= \max\{\|f\|,\, f \in U\}$ and denote $e_1$ function $e_1 \in U$ such that $\|e_1\|= d_1.$ Denote $\Pi_1$ linear subspace generated vector $e_1$.

For $i=2,3,\ldots$ denote
$d_i = \max\{\rho(f,\Pi_{i-1}), f \in U \}$ with $\rho(f,\Pi_{i-1})=\min\{\|f-g\|, g \in \Pi_{i-1} \}$. Define function $e_i$, $e_i \in U$, such that $\rho(e_i,\Pi_{i-1}) = d_i$.
Denote $\Pi_i$ linear subspace generated functions $e_1,\ldots,e_i$.

 For any function $f \in \mathbb{L}_2(0,1)$ denote
$f_{\Pi_i}$ projection of function $f$ on subspace $\Pi_i$ and denote $\tilde f_i = f - f_{\Pi_i}$.

Thus we associate with each $f \in \mathbb{L}_2(0,1)$ sequence of functions $\tilde f_i$, $\tilde f_i \to 0$ as $i \to \infty$.

For the problem of signal detection we say that  set $U$ is {\sl maxiset} for test statistics $T_n$ and $\Xi$ is {\sl maxispace} if the following two statements take place.
\vskip 0.3cm
{\sl i.} any subsequence of alternatives $f_{n_j} \in U$, $cn_j^{-r} < \|f_{n_j}\| < Cn_j^{-r}$, $n_j \to \infty$ as $j \to \infty$, is consistent,.
\vskip 0.3cm
{\sl ii.}  for any  $f \in \mathbb{L}_2(0,1)$, $f \notin \Xi$,  there are sequences  $i_n$ and  $j_{i_n}$ with $i_n \to \infty$, $j_{i_n} \to \infty$ as $n \to \infty$,   such that subsequence $\tilde f_{i_n}$ is inconsistent  and $c j_{i_n}^{-r}<\|\tilde f_{i_n}\| < C j_{i_n}^{-r}$.
\vskip 0.3cm
For problem of hypothesis testing on a density, {\sl ii.} is verified  only for functions $f$ such that $1+ \tilde f_{i}$ are densities for all $i > i_0$.

  We provide proofs of Theorems for definition of maxisets in terms of  subsection \ref{ss2.3}. However it is easy to see that slight modification of this reasoning provide proofs for definition of subsection \ref{ss2.5} as well. Basis $\phi_j$, $1 \le j < \infty$, in subsection \ref{ss2.3} coincides  in this reasoning with basis $e_j$.
  \section{Necessary and sufficient conditions of uniform consistency \label{sec3}}
   We  consider  problem of signal detection in Gaussian white noise discussed in Introduction. Problem is explored in terms of sequence model.

Stochastic differential equation (\ref{q1}) can be rewritten  in terms of  a sequence model based on orthonormal system of functions $\phi_j$, $1 \le j < \infty$, in the following form
\begin{equation}\label{q2}
y_j = \theta_j + \frac{\sigma}{\sqrt{n}} \xi_j, \quad 1 \le j < \infty,
\end{equation}
where $$y_j = \int_0^1 \phi_j dY_n(t), \quad \xi_j = \int_0^1\,\phi_j\,dw(t) \quad \mbox{ and}  \quad \theta_j = \int_0^1 f\,\phi_j\,dt.$$ Denote $\yb =  \{y_j\}_{j=1}^\infty$ and $\thetab = \{\theta_j\}_{j=1}^\infty$.

 We can consider $\thetab$ as a vector in Hilbert space $\mathbb{H}$ with the norm $\|\thetab\| = \Bigl(\sum_{j=1}^\infty \theta_j^2\Bigr)^{1/2}$. We implement the same notation $\| \cdot \|$ in $\mathbb{L}_2$ and in $\mathbb{H}$. Sense of this notation will be always clear from context.

In this notation the problem of hypothesis testing can be rewritten in the following form.
One needs to test the hypothesis
 \begin{equation}\label{sh31}
\mathbb{H}_0 : \thetab = \0b
\end{equation}
 versus alternatives
 \begin{equation}\label{sh32}
 \mathbb{H}_n : \thetab \in V_n =\{\, \thetab : \|\thetab\| \ge \rho_n,\, \thetab \in U,\, U \subset \mathbb{H}\,\}.
\end{equation}
Here $ U $ is bounded convex set.

We say that $\0b = \{0,0,\ldots\}$  is inner point of set $U$ if for any $\yb \in \mathbb{H}$ there is $\lambda > 0$ such that $\lambda \yb \in U$ and $-\lambda \yb \in U$.

 \begin{thm} \label{tqq} Suppose that bounded set U is  convex and $\0b$ is inner point of $U$. Then  there is  sequence $\rho_n  \to 0$ as $n \to \infty$ such that  there is uniformly consistent sequence of tests  for  sets of alternatives $V_n$ with this sequence $\rho_n$,  if and only if, set $U$ is  relatively compact.
 \end{thm}
 If set $U$ is  relatively compact, there  is consistent estimator (see \cite{ih}  and \cite{jo}). Therefore we can choose $\mathbb{L}_2$-norm of consistent   estimator as uniformly consistent test statistics.

 \begin{remark} Suppose $K$ is convex hull of points $\thetab_1, \thetab_2, \ldots$ and $\0b$ is inner point of $K$. Suppose $K$ is not relatively compact and $K \subset U$ where the set $U$ is not necessarily convex. Then, by  Theorem \ref{tqq}, for problem of testing hypothesis (\ref{sh31}) versus alternatives (\ref{sh32}), there does not exist uniformly consistent tests for all sequences $\rho_n  \to 0$ as $n \to \infty$.
 \end{remark}

 Version of Theorem \ref{tqq} holds for problem of testing hypothesis on a density in a following setup. Let $\mathbf{P}$ be probability measure on $\sigma$-algebra $\Im$ defined on  set $D$. Denote $\mathbb{L}_2(\mathbf{P})$ set of measurable functions $f:\,D \to \mathbb{R}^1$ such that
 \begin{equation*}
 \int_S f^2 \, d\,\mathbf{P} < \infty.
 \end{equation*}
 Let
 $X_1,\ldots,X_n$ be i.i.d.r.v.'s having probability measure $\mathbf{Q}$, having density $q = \frac{d\mathbf{Q}}{d\mathbf{P}}$ such that $q \in \mathbb{L}_2(\mathbf{P})$.

 Problem is to test hypothesis $\mathbb{H}_0\,:\, q(s) =  1$ for all $s \in D$ versus alternative $\mathbb{H}_1\,:\, q(s) -  1 \in V_n= \{\, f : \|f\| \ge \rho_n,\, f \in U,\, U \subset \mathbb{L}_2(\mathbf{P})\,\}$. Here $\|f\|$ denotes $\mathbb{L}_2(\mathbf{P})$- norm of function $f$. and $U$  is bounded convex set in $\mathbb{L}_2(\mathbf{P})$

 Define function $\0b(s) = 0$ for all $s \in D$.

 \begin{thm} \label{tqqd} Suppose that  set U is  bounded convex into $\mathbb{L}_2(\mathbf{P})$. Let set $U$ be such that for any function $f \in U$ function $1+f$ is probability density. Let
 $\0b$ be inner point of $U$. Then  there is  sequence $\rho_n  \to 0$ as $n \to \infty$ such that there is uniformly consistent sequence of tests  for  sets of alternatives $V_n$ with this sequence $\rho_n$,  if and only if, set $U$ is  relatively compact.
 \end{thm}
 Reasoning in the proof of Theorem \ref{tqqd} coincides with the reasoning of proof of Theorem \ref{tqq} with unique difference we implement Theorem 4.1 in \cite{er15} instead of Theorem 5.3 in \cite{er15}. We omit this reasoning.

   Similar Theorem holds for problem of signal detection in linear inverse ill-posed problem.

 In Hilbert space $\mathbb{H}$, we observe
  a realization of  Gaussian random vector
 \begin{equation}\label{il1}
  \yb = A\thetab + \epsilon \xib, \quad \epsilon > 0,
  \end{equation}
  where $A: \mathbb{H} \to \mathbb{H}$ is known  linear operator and $\xib$ is Gaussian random vector having known covariance operator $R: \mathbb{H} \to \mathbb{H}$ and
  $\mathbf{E} [\xib] = 0$.

  We explore the same problem of hypothesis testing  $\mathbb{H}_0 : \thetab = 0$ versus alternatives $\mathbb{H}_n\, :\, \thetab \in V_n$.

For any operator $S: \mathbb{H} \to \mathbb{H}$ denote $\frak{R}(S)$ the rangespace of $S$.

  Suppose that the nullspaces of $A$ and $R$ equal  zero and $\frak{R}(A) \subseteq \frak{R}(R^{1/2})$.

\begin{thm}\label{tqq1} Let  operator $R^{-1/2}A$ be bounded. Suppose that bounded set U is  convex and $\0b$ is inner point of $U$. Then the statement of Theorem \ref{tqq} holds.
\end{thm}
\begin{remark} In papers another definition  of uniform consistency is often explored (see, for example, \cite{ing87}). In this definition, (\ref{uc1}) is replaced with  the requirement of existence of sequence of tests $K_n$ such that $\alpha(K_n) \to 0$ and $\beta(K_n,V_n) \to 0$ as $n \to \infty$. By Theorem on exponential decay of type I and type II error probabilities (see  \cite{le73} and  \cite{sch}), the statements of Theorems \ref{tqq} -- \ref{tqq1} for   this definition of consistency follows from  these Theorems.
 \end{remark}
\section{Quadratic test statistics \label{sec4}}
  \subsection{General setup \label{s4.1}}
We  explore  problem of signal detection in Gaussian white noise (\ref{q1}) and (\ref{i30}) with $0 < r < 1/2$ discussed in Introduction. Problem is provided in terms of sequence model (\ref{q2}).

 If $U$ is compact ellipsoid
 \begin{equation*}
 U = \left\{\,\thetab\,:\, \sum_{j=1}^\infty a_j\,\theta_j^2 \le P_0, \thetab = \{\theta_j\}_{j=1}^\infty, \theta_j \in \mathbb{R}^1\right\}
 \end{equation*}
 with $a_j > 0$, $a_j \to \infty$ as $j \to \infty$,
  asymptotically minimax test statistics for sets of alternatives $V_n$ are  quadratic forms
$$
T_n(Y_n) = \sum_{j=1}^\infty \kappa_{nj}^2 y_j^2 - \sigma^2 n^{-1} \rho_n
$$
with some specially defined coefficients $\kappa^2_{nj}$ (see Ermakov \cite{er90}). Here $\rho_n = \sum_{j=1}^\infty \kappa_{nj}^2$.

If coefficients $\kappa_{nj}^2$ satisfy some regularity assumptions,  test statistics $T_n(Y_n)$ are asymptotically minimax (see \cite{er04}) for   wider sets of alternatives
$$
\mathbb{H}_n : f \in \Upsilon_n(R_n,c)  = \{\, f:  R_n(f) > c ,\,\, f \in \mathbb{L}_2(0,1) \,\}
$$
with
$$
R_n(f) = A_n(\thetab) = \sigma^{-4}\,n^{2}\,\sum_{j=1}^\infty\, \kappa_{nj}^2\,\theta_j^2.
$$
and $f = \sum_{j=1}^\infty \theta_j \phi_j$.

A sequence of tests $L_n, \alpha(L_n) = \alpha(1+ o(1))$, $0 <\alpha<1$, is called {\sl asymptotically minimax}  if, for any sequence of tests $K_n$, $\alpha(K_n) \le \alpha,$ there holds
\begin{equation*}
\liminf_{n\to \infty}(\beta\,(K_n,\,\Upsilon_n(R_n,c)) - \beta\,(L_n,\Upsilon_n(R_n,c))) \ge 0.
\end{equation*}
Sequence of test statistics $T_n$ is called asymptotically minimax if   tests generated test statistics $T_n$ are asymptotically minimax.

We make the following  assumptions.

\noindent{\bf A1.} For each $n$  sequence $\kappa^2_{nj}$ is decreasing.

\noindent{\bf A2.} There are positive constants $C_1$ and $C_2$ such that, for each $n$, there holds
\begin{equation}\label{q5}
 C_1 < A_n = \sigma^{-4}\,n^2\,\sum_{j=1}^\infty \kappa_{nj}^4 < C_2.
 \end{equation}

 \noindent{\bf A3.} There are positive constants $c_1$ and $c_2$ such that $c_1n^{-2r} \le \rho_n \le c_2 n^{-2r}$.

  Denote $\kappa_n^2=\kappa^2_{nk_n}$ with $k_n = \sup\,\Bigl\{k: \sum_{j < k} \kappa^2_{nj} \le \frac{1}{2} \rho_n \Bigr\}$.

\noindent{\bf A4.}   There are $C_1$  and $\lambda >1$ such that, for any $\delta > 0$ and for each $n$ we have
 \begin{equation*}
\kappa^2_{[n,(1+\delta)k_n]} < C_1(1 +\delta)^{-\lambda}\kappa_n^2.
\end{equation*}
\noindent{\bf A5.} There holds $\kappa_{n1}^2  \asymp \kappa_n^2$ as $n \to \infty$.  For any $c>1$  there is $C$ such that $\kappa_{n,[ck_n]}^2 \ge C\kappa_n^2$ for all $n$.

\noindent{\sl Example}.  Let
$$
\kappa^2_{nj} = n^{-\lambda}\frac{1}{j^{\gamma}  + c n^\beta}, \quad \gamma >1,
$$
with $\lambda = 2 - 2r -\beta$ and $\beta = (2-4r)\gamma$.
Then A1 -- A5 hold.

Note that A1-A5 imply
\begin{equation}\label{u1}
\kappa_n^4=\kappa^4_{nk_n}  \asymp n^{-2}k_n^{-1}\quad \mbox{and} \quad k_n \asymp n^{2-4r}.
\end{equation}
Theorems \ref{tq3} - \ref{tq8} given below represent  realization of program announced in Introduction.
 \subsection{Analytic form of necessary and sufficient conditions of consistency \label{s4.2}}
 The results will be provided in terms of  Fourier coefficients of functions $f_n = \sum_{j=1}^\infty \theta_{nj} \phi_j$.
\begin{thm}\label{tq3} Assume {\rm A1-A5}. Sequence of alternatives $f_n$, $cn^{-r} \le \|f_n\| \le Cn^{-r}$, is consistent, if and only if, there are $c_1$, $c_2$ and $n_0$  such that there holds
\begin{equation}\label{con2}
\sum_{|j| < c_2k_n} |\theta_{nj}|^2 > c_1 n^{-2r}
\end{equation}
for all $n > n_0$.
\end{thm}
Versions of Theorems \ref{tq3}, \ref{tq4} and \ref{tq6} hold for setups of other sections. In setups of these sections indices $j$ may accept negative values and $\theta_{nj}$ may be complex numbers. By this reason we write $|j|$ instead of $j$ and $|\theta_{nj}|$ instead of $\theta_{nj}$ in (\ref{con2}), (\ref{con3}) and (\ref{con19}).
\begin{thm}  \label{tq4}  Assume {\rm A1-A5}.  Sequence of alternatives $f_n$, $cn^{-r} \le \|f_n\| \le Cn^{-r}$, is inconsistent, if and only if, for any  $c_2$, there holds
\begin{equation}\label{con3}
\sum_{|j| < c_2k_n} |\theta_{nj}|^2 = o( n^{-2r})\quad \mbox{as} \quad {n \to \infty}.
\end{equation}
\end{thm}

Proof of Theorems is based on Theorem \ref{tq2} on asymptotic minimaxity of test statistics $T_n$.

Define sequence of tests $K_n(Y_n) = {\bf 1}_{\{n^{-1}T_n(Y_n) > (2A_n)^{1/2} x_\alpha\}}$, $0 < \alpha <1$, where $x_\alpha$ is defined by the equation $\alpha = 1 - \Phi(x_\alpha)$.
\begin{thm}\label{tq2}
Assume {\rm A1-A5}. Then sequence of tests $K_n(Y_n)$ is asymptotically minimax for the sets  $\Upsilon_n(R_n,c)$ of alternatives.
There hold $\alpha(K_n) = \alpha + o(1)$ and
\begin{equation}\label{aq1}
\beta(K_n,f_n) = \Phi(x_{\alpha} - R_n(f_n)(2A_n)^{-1/2})(1+o(1))
\end{equation}
uniformly onto all  sequences $f_n$ such that $R_n(f_n)< C$ for any $C > 0$.
\end{thm}
A version of Theorem \ref{tq2} for the problem of signal detection
with  heteroscedastic white noise  has been proved in \cite{er03}.

Such a form of conditions in Theorems \ref{tq3} and \ref{tq4} can be explained by concentration of coefficients $\kappa_{nj}^2$  in zone $j = O(k_n)$ for test statistics $T_n$ and  for $A_n(\thetab_n)$.

Version of Theorem \ref{tq2} for problem of hypothesis testing on distribution function provides necessary and sufficient conditions of uniform consistency of sets of alternatives defined in terms of distribution functions.
\subsection{Maxisets. Qualitative structure of consistent sequences of alternatives}
Denote $s = \frac{r}{2 -4r}$. Then $r = \frac{2s}{1 + 4s}$.
\begin{thm}\label{tq1} Assume {\rm A1-A5}. Then balls  $\mathbb{\bar B}^s_{2\infty}(P_0)$, $P_0 > 0$, are maxisets for  test statistics $T_n(Y_n)$. \end{thm}
  For maxisets $\mathbb{\bar B}^s_{2\infty}(P_0)$ with  deleted "small" $\mathbb{L}_2$- ball asymptotically minimax tests  have been found in \cite{er18}. In \cite{ing02}, similar result has been obtained   for Besov bodies in $\mathbb{ B}^s_{2\infty}$ defined in terms of  wavelets  coefficients.
\begin{thm}\label{tq7} Assume {\rm A1-A5}. Then
 sequence of alternatives $f_n$, $c n^{-r}\le \|f_{n}\| \le C n^{-r}
$, is consistent, if and only if,  there are maxiset  $\mathbb{\bar B}^s_{2\infty}(P_0)$, $P_0 > 0$,  and sequence  $f_{1n} \in \mathbb{\bar B}^s_{2\infty}(P_0)$, $c_1 n^{-r}\le \|f_{1n}\| \le C_1 n^{-r}
$, such that $f_{1n}$ is  orthogonal to $f_n - f_{1n}$, that is, there holds
\begin{equation}
\label{ma1}
\| f_n \|^2 =  \| f_{1n}\|^2  + \|f_n - f_{1n}\|^2,
\end{equation}
Therefore, if we have maxiset $\mathbb{\bar B}^s_{2\infty}(P_0)$, $P_0 > 0$,  sequence  of arbitrary functions $f_{1n} \in \mathbb{\bar B}^s_{2\infty}(P_0)$, $c_1 n^{-r}\le \|f_{1n}\| \le C_1 n^{-r}$ and sequence of arbitrary functions $f_{2n} $, $c_1 n^{-r}\le \|f_{2n}\| \le C_1 n^{-r}$ orthogonal to $f_{1n}$, then sequence of simple alternatives $f_n = f_{1n} + f_{2n}$ is consistent.
\end{thm}
\begin{thm}\label{tq11} Assume {\rm A1-A5}. Then,  for any $\varepsilon > 0$, for any consistent sequence of alternatives $f_n$,   $c n^{-r}\le \|f_{n}\| \le Cn^{-r}
$  there are   maxiset  $\mathbb{\bar B}^s_{2\infty}(P_0)$, $P_0 > 0$, and sequence of functions $f_{1n}$ , $c_1 n^{-r}\le \|f_{1n}\| \le C_1n^{-r}
$, belonging to maxiset $\mathbb{\bar B}^s_{2\infty}(P_0)$  such that there holds
\vskip 0.2cm
function $f_{1n}$ is orthogonal to $f_n - f_{1n}$
\vskip 0.2cm
for any $\alpha$, $0 < \alpha < 1$, for the tests $K_n$, $\alpha(K_n) = \alpha (1+ o(1))$ as $n \to \infty$, there is $n_\varepsilon$  such that,
for any $n > n_\varepsilon$, there hold
\begin{equation}
\label{uuu}
|\beta(K_n,f_n) - \beta(K_n,f_{1n})| \le \varepsilon
\end{equation}
and
\begin{equation}
\label{uu1}
\beta(K_n,f_n-f_{1n})  \ge 1 - \alpha - \varepsilon.
\end{equation}
\end{thm}
If functions $f_n = \sum_{j=1}^\infty\, \theta_{nj} \phi_j$ satisfy $c_1\,n^{-r} \le \|f_n\| \le C_1\,n^{-r}$, then for any $c$ there is $P_0$ such that $f_{1n} = \sum_{j=1}^{[ck_n]} \theta_{nj} \phi_j \in \mathbb{\bar B}^s_{2\infty}(P_0)$   (see Lemma \ref{ld3}). Since coefficients $\kappa^2_{nj}$, $j > ck_n$, are relatively small for large  $c$, this allows to prove Theorems \ref{tq7} and \ref{tq11}.

Maxisets $\mathbb{\bar B}^s_{2\infty}(P_0)$, $P_0 > 0$ in Theorems \ref{tq7}, \ref{tq11} and \ref{tq12} can be replaced with arbitrary maxiset $ U$.
\subsection{Interaction of consistent and inconsistent sequences of alternatives. Purely consistent sequences}
\begin{thm}  \label{tq5}  Assume {\rm A1-A5}. Let sequence of alternatives $f_n$ be consistent. Then, for any inconsistent sequence of alternatives $f_{1n}$, for tests $K_n$, $\alpha(K_n) = \alpha (1 + o(1))$, $0 < \alpha <1$, generated test statistics $T_n$, there holds
\begin{equation*}
\lim_{n \to \infty} (\beta(K_n,f_n) - \beta(K_n,f_n + f_{1n})) = 0.
\end{equation*}
\end{thm}
\begin{thm}\label{tq6} Assume {\rm A1-A5}. Sequence of alternatives $f_n$, $cn^{-r} \le \|f_n\| \le Cn^{-r}$, is purely $n^{-r}$-consistent, if and only if, for any $\varepsilon >0$, there is $C_1= C_1(\varepsilon)$ such that there holds
\begin{equation}\label{con19}
\sum_{|j| > C_1k_n} |\theta_{nj}|^2 \le \varepsilon n^{-2r}
\end{equation}
for all $n> n_0(\varepsilon)$.
\end{thm}
\begin{thm}\label{tq12} Assume {\rm A1-A5}. Then sequence $f_n$, $c n^{-r}\le \|f_{n}\| \le C n^{-r}
$, is  purely $n^{-r}$-consistent, if and only if, for any $\varepsilon > 0$,  there is $\gamma_\epsilon$ and sequence of functions $f_{1n}$  belonging to maxiset $\mathbb{\bar B}^s_{2\infty}(\gamma_\epsilon)$ such that  $\|f_n - f_{1n}\| \le \varepsilon n^{-r}$ and (\ref{ma1}) holds.
\end{thm}
\begin{thm}\label{tq8} Assume {\rm A1-A5}. Then sequence of alternatives $f_n$, $cn^{-r} < \|f_n\| < Cn^{-r}$, is purely $n^{-r}$-consistent, if and only if, for any  inconsistent subsequence of alternatives $f_{1n_i}$, ,  $cn_i^{-r} < \|f_{1n_i}\| < Cn_i^{-r}$, there holds
\begin{equation}\label{ma2}
\|f_{n_i} + f_{1n_i}\|^2 =  \|f_{n_i} \|^2 + \| f_{1n_i}\|^2 + o(n_i^{-r}),
\end{equation}
where $n_i \to \infty$ as  $i \to \infty$.
\end{thm}
\begin{remark}\label{rem1}{\rm Let $\kappa_{nj}^2 > 0$ for $j \le l_n$  and  let $\kappa^2_{nj} = 0$ for $j > l_n$ with $l_n \asymp n^{2-4r}$ as $n \to \infty$.
Analysis of  proofs of Theorems    shows that Theorems \ref{tq3} - \ref{tq8}  remain valid for this setup  if A4 and A5 are replaced with
\vskip 0.25cm
\noindent{\bf A6.} For any $c$, $0 <c <1$, there is $c_1$  such that $\kappa^2_{n,[cl_n]} \ge c_1 \kappa^2_{n1}$ for all $n$.
\vskip 0.25cm
In all corresponding  reasoning we should put $\kappa^2_n = \kappa_{n1}^2$ and $k_n = l_n$.

Theorems   \ref{tq4} and \ref{tq6} hold with the following changes. It suffices to put $c_2 < 1$  in Theorem \ref{tq4} and to take $C_1(\epsilon) < 1$ in Theorem \ref{tq6}.

Proof of corresponding versions of Theorems \ref{tq3} - \ref{tq8} is obtained by simplification of provided reasoning and is omitted.}
\end{remark}
\section{Kernel-based tests \label{sec5}}
We  continue to explore  problem (\ref{i29}) and (\ref{i30})  of signal detection in Gaussian white noise  with $0< r < 1/2$. We suppose additionally that signal $f$ belongs to $\mathbb{L}_2^{per}(\mathbb{R}^1)$ the set of 1-periodic functions such that $f(t) \in \mathbb{L}_2(0,1)$. This allows to extend our model on real line $\mathbb{R}^1$ putting $w(t+j) = w(t)$ for all integer $j$ and $t \in [0,1)$ and  to write the forthcoming integrals over all real line.

Define kernel estimator
\begin{equation}\label{yy}
\hat{f}_n(t) = \frac{1}{h_n} \int_{-\infty}^{\infty} K\Bigl(\frac{t-u}{h_n}\Bigr)\, d\,Y_n(u), \quad t \in (0,1),
\end{equation}
where $h_n > 0$, $h_n \to 0$ as $n \to \infty$.

The kernel $K$ is bounded function such that the support of $K$ is contained in $[-1,1]$, $K(t) = K(-t)$ for $t \in \mathbb{R}^1$ and $\int_{-\infty}^\infty K(t)\,dt = 1$.

Denote $K_h(t) = \frac{1}{h} K\Bigl(\frac{t}{h}\Bigr)$, $t \in \mathbb{R}^1$ and $h >0$.

In (\ref{yy}) we supposed that, for any $v$, $0 <v  < 1$, we have
\begin{equation*}
 \int_{1}^{1+v} K_{h_n}(t-u)\,dY_n(u) = \int_{0}^{v} K_{h_n}(t-1-u)\,f(u)\,du
 + \frac{\sigma}{\sqrt{n}} \int_{0}^{v} K_{h_n}(t-1-u)\, dw(u)
\end{equation*}
and
\begin{equation*}
 \int_{-v}^{0} K_{h_n}(t-u)\,dY_n(u) =  \int_{1-v}^{1} K_{h_n}(t-u+1)\,f(u)\,du
 + \frac{\sigma}{\sqrt{n} } \int_{1-v}^{1} K_{h_n}(t-u+1)\, dw(u).
\end{equation*}
Define kernel-based   test statistics $$T_n(Y_n) = T_{nh_n}(Y_n) =nh_n^{1/2}\sigma^{-2}\gamma^{-1} (\|\hat f_{n}\|^2- \sigma^2(nh_n)^{-1}\|K\|^2),$$  where
$$
\gamma^2 = 2 \int_{-\infty}^\infty \Bigl(\int_{-\infty}^\infty K(t-s)K(s) ds\Bigr)^2\,dt.
$$
We call sequence of alternatives $f_n$, $cn^{-r} \le \|f_n\| \le Cn^{-r}$, $n^{-r}$-{\sl consistent} if, there is constant $c_1$ such that (\ref{vas1}) holds for any tests $K_n$, $\alpha(K_n) = \alpha\,(1 + o(1))$. $0 < \alpha <1$, generated sequence of test statistics $T_n$ with  $h_n < c_1 n^{4r-2}$, $h_n \asymp n^{4r-2}$.

 We call sequence of alternatives $f_n$, $cn^{-r} \le \|f_n\| \le Cn^{-r}$, $n^{-r}$-{\sl inconsistent} if sequence of alternatives $f_n$ is  inconsistent for any tests generated arbitrary test statistics $T_n$ with $h_n \to 0$ as $n \to \infty$.

Problem will be explored in terms of sequence model.

Let we observe a realization of random process $Y_n(t)$ with $f= f_n$.

 For $-\infty < j < \infty$, denote
$$
\hat K(jh) = \int_{-1}^1 \exp\{2\pi ijt\}\,K_h(t)\, dt,\quad h > 0,
$$
$$
y_{nj} = \int_0^1 \exp\{2\pi ijt\}\, dY_n(t),
\quad
\xi_j = \int_0^1 \exp\{2\pi ijt\}\, dw(t),
$$
$$
\theta_{nj} = \int_0^1 \exp\{2\pi ijt\}\, f_n(t)\, dt.
$$
In this notation we can write kernel estimator in the following form
\begin{equation}\label{au1}
\hat \theta_{nj} = \hat K(jh_n)\, y_{nj} =
\hat K(jh_n)\, \theta_{nj} + \sigma\, n^{-1/2}\, \hat K(jh_n)\, \xi_j, \quad -\infty < j < \infty,
\end{equation}
and test statistics $T_n$ admit the following representation
\begin{equation}\label{au111}
T_n(Y_n) = nh_n^{1/2}\sigma^{-2}\gamma^{-1} \Bigl(\sum_{j=-\infty}^\infty |\hat \theta_{nj}|^2  -  n^{-1}\sigma^2 \sum_{j=-\infty}^\infty |\hat K(jh_n)|^2\Bigr).
\end{equation}
If we put $|\hat K(jh_n)|^2 = \kappa^2_{nj}$, we get that definitions of test statistics $T_n(Y_n)$ in this section and in sections   \ref{sec4} are almost coincide. The setup of section  \ref{sec5} differs from setup of section \ref{sec4} only heteroscedastic white noise. Another difference  in the setup is that the function $\hat K(\omega)$, $\omega \in \mathbb{R}^1$, may have zeros.  Since  differences are insignificant the same results are valid.
Denote $k_n = [n^{2-4r}]$.

\begin{thm}\label{tk3}  The statements of Theorems \ref{tq3}, \ref{tq4}, \ref{tq7}-\ref{tq8} hold for this setup as well. The statement of Theorem \ref{tq1} holds also with $ \mathbb{\bar B}^s_{2\infty}$  replaced with $\mathbb{B}^s_{2\infty}$.
\end{thm}
In version of Theorem \ref{tq1},  {\sl ii.} in definition of maxisets holds for test statistics $T_n$ having arbitrary values $h_n > 0$, $h_n \to 0$ as $n \to \infty$.

Denote
$$
T_{1n}(f) = T_{1n}(f,h_n) =\int_0^1\Bigl(\frac{1}{h_n}\int K\Bigl(\frac{t-s}{h_n}\Bigr)f(s)\, ds\Bigr)^2 dt.
$$
For sequence $\rho_n >0$, define  sets
$$
\Upsilon_{nh_n}(T_{1n},\rho_n) = \{f: T_{1n}(f) > \rho_n,\, f \in  \mathbb{L}_2^{per}(\mathbb{R}^1)\}.
$$

Define sequence of kernel-based tests $K_n ={\bf 1}_{\{T_n(Y_n) \ge x_\alpha\}}$,  $0 < \alpha <1$, with $x_\alpha$  defined the equation $\alpha = 1 - \Phi(x_\alpha)$.

Proof of Theorems is based on the following Theorem \ref{tk2} on asymptotic minimaxity of kernel-based tests $K_n$ (see Theorem 2.1.1  in  \cite{er03}).
\begin{thm}\label{tk2} Let $h_n^{-1/2}n^{-1} \to 0$, $h_n \to 0$ as $n \to \infty$.
Let
\begin{equation*}
0 < \liminf_{n \to \infty}  n \rho_n h_n^{1/2} \le \limsup_{n \to \infty} n\rho_nh_n^{1/2} < \infty.
\end{equation*}
Then sequence  of kernel-based tests $K_n$, is asymptotically minimax for the sets of alternatives $\Upsilon_{nh_n}(T_{1n},\rho_n)$.
There hold $\alpha(L_n) = \alpha(1 + o(1))$ and
\begin{equation}\label{33}
\beta(K_n,f_n)  = \Phi(x_\alpha - \gamma^{-1} \sigma^{-2}n h_n^{1/2} T_{1n}(f_n))(1 + o(1))
\end{equation}
uniformly onto sequences $f_n\in \mathbb{L}_2^{per}(R^1)$ such that $n h_n^{1/2} T_{1n}(f_n) < C $.
\end{thm}
We have
\begin{equation}\label{z33}
T_{1n}(f_n)   = \sum_{j=-\infty}^\infty |\hat K(jh_n)|^2 |\theta_{nj}|^2.
\end{equation}
Note that the unique difference  of setups of Theorems \ref{tk2} and \ref{tq2} is heteroscedastic noise. Thus roof of Theorem \ref{tk2} can be obtained  by easy modification of the proof of Theorem \ref{tq2}.
\section{ $\chi^2$-tests \label{sec6}}
Let $X_1,\ldots,X_n$ be i.i.d.r.v.'s having c.d.f. $F \in \Im$.
Let c.d.f. $F(x)$ have a density $1 + f(x) = dF(x)/dx$, $x \in (0,1)$, $f \in L_2^{per}(0,1)$.

We explore the problem of testing hypothesis (\ref{i29}) versus alternatives (\ref{i30}) with $0< r < 1/2$ discussed in Introduction.

For any sequence $m_n$, denote  $\hat p_{nj} = \hat F_n(j/m_n) - \hat F_n((j-1)/m_n)$, $1 \le j \le m_n$.

Test statistics of $\chi^2$-tests equal
$$
T_n(\hat F_n) = n\, m_n \, \sum_{j=1}^{m_n}\, (\hat p_{nj}  - 1/m_n)^2.
$$
Let
$$ f_n = \sum_{j=-\infty}^\infty \theta_{nj} \phi_j, \quad \phi_j(x) = \exp\{2\pi i\,j\,x\,\}, \quad x \in (0,1).
$$
We call sequence of alternatives $f_n$, $cn^{-r} \le \|f_n\| \le Cn^{-r}$, $n^{-r}$-{\sl consistent}, if there is $c_1$ such that, (\ref{vas1}) holds for any tests $K_n$, $\alpha(K_n) = \alpha\,(1 + o(1))$. $0 < \alpha <1$, generated sequence of chi-squared test statistics $T_n$ with number of cells $m_n > c_1 n^{2-4r}$, $m_n \asymp n^{2-4r}$.

We call sequence of alternatives $f_n$, $cn^{-r} \le \|f_n\| \le Cn^{-r}$, $n^{-r}$-{\sl inconsistent} if sequence of alternatives $f_n$ is  inconsistent for all tests generated  test statistics $T_n$ having number of cells $m_n$, $m_n \to \infty$  as $n \to \infty$..

Denote $k_n = \Bigl[n^\frac{2}{1 + 4s}\Bigr]\asymp n^{2-4r}$.

The differences in versions of Theorems \ref{tq3} --\ref{tq8} for this setup are caused only the requirement that functions $f_n$, $f_{1n}$ and $f_{2n}$ should be densities.
\begin{thm}\label{tchi3}  The statements of Theorems \ref{tq3}, \ref{tq4} and \ref{tq1}-\ref{tq11}, \ref{tq6}-\ref{tq8} hold for this setup with the following differences.

 In version of Theorem \ref{tq1} balls $ \mathbb{\bar B}^s_{2\infty}$  is replaced with bodies $\mathbb{\tilde B}^s_{2\infty}$.

 In version  of  Theorem \ref{tq1},   {\sl ii.} in definition of maxisets holds for test statistics $T_n$ with arbitrary choice of number of cells $m_n$, $m_n \to \infty$ as $n \to \infty$.

In version of Theorem \ref{tq11}  we consider only sequences of alternatives $f_n$ such that the following assumption holds.

{\bf B.} There is $c_0$ such that, for all $c > c_0$, functions
$$
1 + f_{cn} =  1 + \sum_{|j| > c m_n} \theta_j \phi_j \quad \mbox{and} \quad 1 + f_n - f_{cn} =  1 + \sum_{|j| < c m_n} \theta_j \phi_j
$$
are densities.

We implement definition of purely consistent sequences only for sequences $f_n$ satisfying {\rm B}.
\end{thm}
In proof of version of Theorem  \ref{tq11} for chi-squared tests, we show that there is $C_\varepsilon = C(\varepsilon,c,C,c_0)$ such that, for densities $1 + f_{1n} =  1 + \sum_{|j| < C_\varepsilon m_n} \theta_j \phi_j$, (\ref{ma1}), (\ref{uuu}) and (\ref{con19}) hold. By Lemma \ref{ld3} given below, there is $\gamma_\varepsilon$ such that $f_{1n} \in \gamma_\varepsilon U$.

  In Theorem \ref{tchi5}, given bellow, definitions of consistency and inconsistency proposed in subsection \ref{ss2.1} are treated if simple alternatives $f_n$ are replaced with  distribution functions $F_n$ and hypothesis is  $\mathbb{H}_0 : F(x) = F_0(x) =x$, $x \in[0,1]$.
\begin{thm}  \label{tchi5}   Let sequence of alternatives   $F_n$ be  consistent. Let $F_{1n}$ be inconsistent sequence of alternatives such that   $F_{2n} = F_n(x) + F_{1n}(x) - F_0(x)$  are distribution functions. Then for tests $K_n$, $\alpha(K_n) = \alpha (1 + o(1))$, $0 < \alpha < 1$, there holds
\begin{equation*}
\lim_{n \to \infty} (\,\beta (K_n,F_n) - \beta (K_n,F_{2n})\,) = 0.
\end{equation*}
\end{thm}
Proof of Theorems are based on the following Theorem \ref{chi2} on asymptotic minimaxity of chi-squared tests given below. Theorem \ref{chi2} is  summary of results of Theorems 2.1 and 2.4 in \cite{er97}.

For sequence $\rho_n > 0$, define  sets of alternatives
$$
\Upsilon_n(T_n,\rho_n) = \Bigl\{\, F: T_n(F) \ge \rho_n,\, F \in \Im\, \Bigr\}.
$$

The definition of asymptotic minimaxity of tests is the same as in section \ref{sec4}.

Define the tests
$$
K_n = {\bf 1}_{\{2^{-1/2}m_n^{-1/2}(T_n(\hat F_n) - m_n + 1) > x_\alpha\}}
$$
where $x_\alpha$ is defined the equation $\alpha = 1- \Phi(x_\alpha)$.
\begin{thm}\label{chi2} Let $m_n \to \infty$, $m_n^{-1} n^2 \to \infty$ as $n \to \infty$. Let
\begin{equation*}
0 <\, \liminf_{n \to \infty}\, m_n^{-1/2}\rho_n \le \limsup_{n \to \infty}\, m_n^{-1/2} \rho_n < \, \infty.
\end{equation*}
Then $\chi^2$-tests  $K_n$, $\alpha(K_n) = \alpha + o(1)$, $0 < \alpha <1$, are asymptotically minimax for the sets of alternatives $\Upsilon_n(T_n,\rho_n)$.
There holds
\begin{equation*}
\beta(K_n,F_n) = \Phi(x_\alpha - 2^{-1/2}m_n^{-1/2}T_n(F_n))(1 +o(1))
\end{equation*}
uniformly onto sequences $F_n$ such that $ T_n(F_n) \le C m_n^{1/2} $.
\end{thm}
Note that for implementation of Theorem \ref{chi2} to proof of Theorems       \ref{tchi3} and \ref{tchi5} we need to make a transition from indicator functions to trigonometric functions. Such a transition is realized in Appendix.
\section{ Cramer -- von Mises tests \label{sec7}}
We  consider Cramer -- von Mises test statistics as functional
$$
T^2(\hat{F}_n -F_0) = \int_0^1 (\hat{F}_n(x) - F_0(x))^2\, d F_0(x)
$$
depending on empirical distribution function $\hat F_n$.
Here $F_0(x)=x$, $x \in [0,1]$.

Denote $K_n= K_n(X_1, \ldots, X_n)$ sequence of Cramer- von Mises tests.

 A part of further results holds for setup (\ref{ax1}) and (\ref{ax2})  with $\Upsilon_n = \Upsilon_n(a) \doteq \Upsilon_n(T^2, a n^{-1})$, $a >0$.

  We say that Cramer - von Mises test is asymptotically unbiased if, for any $a > 0$, for any $\alpha$, $0 < \alpha < 1$, for tests $K_n$, $\alpha(K_n) = \alpha + o(1)$, there holds
\begin{equation}\label{dur}
 \limsup_{n \to \infty}\sup_{F \in \Upsilon_n(a)} \beta_F(K_n) < 1 -\alpha.
\end{equation}
 Nonparametric tests satisfying (\ref{dur}) are called also uniformly consistent (see Ch. 14.2 in \cite{le}).

Proof of results is based on the following Theorem \ref{tcm}.

 \begin{thm}\label{tcm} The following three statements hold.

 {\sl i.} For sequence of alternatives $F_n$, there is sequence of Cramer - von Mises tests $K_n$
 such that
 \begin{equation}\label{cru0}
 \lim_{n \to \infty} (\alpha(K_n) + \beta_{F_n}(K_n)) = 0,
 \end{equation}
 holds, if and only if, there holds
 \begin{equation} \label{cru}
 \lim_{n \to \infty} n\, T^2(F_n - F_0) = \infty.
 \end{equation}

 {\sl ii.} Cramer - von Mises tests are asymptotically unbiased.

 {\sl iii.} For any sequence of Cramer - von Mises tests $K_n$,
 \begin{equation*}
 \lim_{n \to \infty} (\alpha(K_n) + \beta_{F_n}(K_n)) \ge 1,
 \end{equation*}
 holds, iff, there holds
 \begin{equation*}
 \lim_{n \to \infty} n\, T^2(F_n - F_0) = 0.
 \end{equation*}
 \end{thm}
 Sufficiency in {\sl i.} and {\sl iii.} in Theorem \ref{tcm} is wellknown (see  \cite{ing87}).
Necessary conditions in {\sl i.} and in {\sl iii.}  follows  easily  from {\sl ii.}

From now on we explore the problem of testing hypothesis (\ref{i29}) versus alternatives (\ref{i30}) with $0< r < 1/2$ discussed in Introduction.

 If c.d.f. $F$ has density, we can write the functional $T^2(F-F_0)$ in the following form (see Ch.5,  \cite{wel})
\begin{equation*}
T^2(F-F_0) = \int_0^1\int_0^1(\min\{s,t\} - st)\, f(t)\, f(s)\, ds\, dt
\end{equation*}
with $f(t) = d(F(t) - F_0(t))/dt$.

If we consider the orthonormal expansion of function
\begin{equation*}
f(t) =  \sum_{j=1}^\infty \theta_j \phi_j(t)
\end{equation*}
on trigonometric basis $\phi_j(t) = \sqrt{2}\, \cos(\pi j t)$, $1 \le j <\infty$, then we get
\begin{equation}\label{om3}
nT^2(F-F_0) = n\,\sum_{j=1}^\infty\frac{\theta_j^2}{\pi^2 j^2}.
\end{equation}
Denote $k_n = [n^{(1-2r)/2}]$.

In Theorems \ref{tom1}  and \ref{tom4} given below, we follow the definition of consistency provided in subsection \ref{ss2.1}.
 \begin{thm}\label{tom1} For orthonormal system of functions $\phi_j(t) = \sqrt{2}  \cos(\pi j t)$, $t \in [0,1)$, $j = 1,2, \ldots$,
the bodies  $\mathbb{\bar B}^s_{2\infty}(P_0)$ with $s = \frac{2r}{1 - 2r}$,  $r = \frac{s}{2+2s}$, are maxisets for Cramer -- von Mises test statistics.
\end{thm}
In previous sections functionals $T_n$ depend on $n$. In this setup we explore the unique functional $T$ for all $n$ and for different values of $r$, $0 < r < 1/2$. To separate the study of sequences of alternatives for different $r$, we consider for fixed $r$ only sequences of alternatives satisfying G1.

{\bf G1.} For any  $\varepsilon > 0$   there is $c_\epsilon$ such that there holds
\begin{equation*}
n\sum_{|j| < c_\epsilon k_n} \theta_{nj}^2 j^{-2} < \varepsilon
\end{equation*}
for all $n  > n_0(\varepsilon,c_\epsilon)$.

If  G1 does not hold for some  $c_\epsilon = c_n \to 0$, $c_n k_n \to \infty$ as $n \to \infty$ and functions $ 1 + \bar f_n = 1 + \sum_{j < c_n k_n} \theta_{nj}\,\phi_j$ are densities, then  (\ref{vas1}) holds for some sequence of functions $\bar f_n$, $\|\, \bar f_n \| = o(n^{-r})$. Thus this case of consistency can be studied in the framework of the faster rate of convergence of sequence of alternatives.

\begin{thm}  \label{tom4} Let sequence of alternatives $f_n$ satisfies \mbox{\rm G1}. Then for sequence $f_n$ the statements of Theorems \ref{tq3}, \ref{tq4}, \ref{tq7}, \ref{tq11}, \ref{tq6} and \ref{tq8}  are valid with the following changes.

In version of Theorem \ref{tq11} it is supposed that B holds.

 In Theorem \ref{tom4} definition of pure consistency is considered for sequences of functions $f_n$ satisfying \mbox{\rm B}.
\end{thm}
\begin{thm}\label{tom6} The statement of Theorem \ref{tchi5} holds for this setup as well.
\end{thm}
\section{$n^{-1/2}$-- rate of convergence \label{sec8}}
In section we extend results of sections \ref{sec4} -- \ref{sec7} to the case $r =1/2$.
We show that, for $r= 1/2$,   the sets
\begin{equation*}
U(l,P_0)) =  \{f : f= \sum_{j=1}^\infty \theta_j \, \phi_j,\, \| f\| \le P_0, \, f \in \mathbb{L}_2(0,1)\}
\end{equation*}
with  $l=1,2,\ldots$ and $P_0 > 0$ and  the linear space
$\Xi = \{f\,:\, f \in  U(l,P_0)$ for some integer $l \, \mbox{and}\, P_0 > 0\}$ satisfy {\it i.} and {\it ii.} respectively in definition of maxisets. Moreover sets $U(l,P_0)$ can replace with maxisets in versions of Theorems \ref{tq7}, \ref{tq11} and \ref{tq12}.

Problems of hypothesis testing in sections \ref{sec4}, \ref{sec5} and \ref{sec7} are covered  the following setup.

We observe sequence of independent random variables $y_j = \theta_j + n^{-1/2} \sigma_j\, \xi_j$ where $\xi_j$, $1 \le j < \infty$, are Gaussian random variables, $\mathbf{E} \xi_j =0$ and $\mathbf{E} [\xi_j^2] = 1$.

Define functional
\begin{equation*}
T(\thetab) = \sum_{j=1}^\infty \kappa_j^2\theta_j^2,\quad \thetab = \{\theta_j\}_1^\infty,
\end{equation*}
where coefficients $\kappa_j^2$ satisfy the following conditions.
\vskip 0.15cm
{\bf D1}. Sequence $\kappa_j^2$ is decreasing and $\sum_{j=1}^\infty \kappa_j^2 < \infty$.
\vskip 0.15cm
{\bf D2} There is $C >0$ such that $0 < \sigma_j < C$ for all $1 \le j < \infty$.
\vskip 0.15cm
Problem is to test hypothesis
\begin{equation}\label{hy81}
\mathbb{H}_0\,:\, \theta_j =0, \quad 1 \le j < \infty
\end{equation}
versus alternatives
\begin{equation}\label{hy82}
\mathbb{H}_n \,:\, \theta_j = \theta_{nj}, \quad 1 \le j < \infty,
\end{equation}
where $T(\thetab_n) \asymp n^{-1}$ with $\thetab_n = \{\theta_{nj}\}_1^\infty$.
\begin{thm}\label{th80} For $r =1/2$ sets $U(l,P_0)$, $l =1,2,\ldots$, $P_0 >0$, and linear space $\Xi$ satisfy {\it i.} and {\it ii.} respectively in definition of maxisets.
\end{thm}

\begin{thm}\label{th81} Assume D1 and D2. Then Theorems \ref{tq3}, \ref{tq4} and \ref{tq1} - \ref{tq8} are valid with $k_n =1$ and sets $\mathbb{\bar B}_{2\infty}^s(P_0)$  replaced with sets $U(l,P_0)$, where $l =1,2,\ldots$ and $P_0 >0$.
\end{thm}
Proof of Theorem \ref{th81} is based on Theorem \ref{th82} given below and evident modification of {\sl iii.} in Theorem \ref{tcm} on this setup. Reasoning are akin to proof of Theorems in section \ref{sec4} and is omitted. Note only that for verifying {\sl ii} in definition of maxisets we put $\bar f_l = \sum_{j=l}^\infty \theta_j \phi_j$ (see proof of Theorem \ref{tq1}). After that we implement version of Theorem \ref{tq4} for this setup.

Denote $z_j = n^{1/2} y_j$ and $\eta_j = n^{1/2}\theta_j$. Then problem of hypothesis testing (\ref{hy81}) and (\ref{hy82}) is replaced with the following.

We observe independent random variables $z_j = \eta_j + \sigma_j \xi_j$. Problem is to test hypothesis
\begin{equation}\label{hy83}
\mathbb{H}_0\,:\, \eta_j =0, \quad 1 \le j < \infty
\end{equation}
versus alternatives
\begin{equation}\label{hy84}
\mathbb{H}_n \,:\, \eta_j = \tau_{j}, \quad 1 \le j < \infty,
\end{equation}
where $0<T(\taub) < \infty$ with $\taub = \{\tau_{j}\}_1^\infty$.

For $a > 0$, define sets of alternatives
\begin{equation}\
\Upsilon(a) = \{ \,\etab\,:\, T(\etab) > a, \etab = \{\eta_j\}_1^\infty, \, \eta_j \in \mathbb{R}^1\}
\end{equation}
We say that test $K$ is unbiased \cite{le}, if
\begin{equation}\label{hy85}
\alpha(K) + \beta(K,\Upsilon(a))  < 1.
\end{equation}
Denote $\zb = \{z_j\}_1^\infty$.
\begin{thm}\label{th82} Assume D1 and D2. Then tests $K$, $\alpha(K) = \alpha$, $0 < \alpha <1$, generated test statistics $T(\zb)$ are unbiased.
\end{thm}
Proof of Theorem \ref{th82} is provided in \ref{subsec9.7}.

For chi-squared tests with  number of cells $m_n = m=$const similar Theorem holds for $r =1/2$ with the same definition of consistency as in section \ref{sec6}.
\begin{thm}\label{th83} For $r = 1/2$, for chi-squared tests Theorem \ref{th80} holds as well. Statement of Theorems \ref{tq3}, \ref{tq4} and \ref{tq7}, \ref{tq11}, \ref{tq6}-\ref{tq8} hold with the same changes as in Theorem \ref{tchi3} and with $k_n=1$.

  Emphasize that Besov bodies $\mathbb{\tilde B}^s_{2\infty}(P_0)$ in versions of Theorems  \ref{tq7}, \ref{tq11} and \ref{tq12}     are replaced with sets $U(l,P_0)$, $l =1,2,\ldots$ and $P_0 >0$.
\end{thm}
For proof of Theorem \ref{th83} we implement wellknown  fact that $n T_n(F_n) > c$ is necessary and sufficient condition for consistency of sequence of alternatives $F_n \in \Im$ for chi-squared tests with fixed number of cells.

Theorem \ref{th82} allows to obtain versions of Theorems \ref{tom1}--\ref{tom6} for problem of hypothesis testing \ref{hy83} and \ref{hy84} with test statistics $T$ having $\kappa_j^2 \asymp j^{-2\lambda}$, $2\lambda > 1$.

Such a setup arises in particular for test statistics $T$ constructed on the base of technique of reproducing kernel Hilbert spaces \cite{gr}.
\begin{thm} \label{th84} Let $\kappa_j^2 \asymp j^{-2\lambda}$, $2\lambda >1$. Then statements of Theorems \ref{tom1} -- \ref{tom6}  holds with $s = \frac{2\lambda r}{1 -2 \lambda r}$ and $ k_n \asymp n^{\frac{1-2r}{2\lambda}}$ as $n \to  \infty$. All assumptions caused the requirement of density non-negativity are omitted.
\end{thm}
Proof of Theorem \ref{th84} is akin to proof of Theorems \ref{tom1}--\ref{tom6} and is omitted.

\appendix

\section{Proof of Theorems}\label{app}

 \subsection{Proof of Theorems of section \ref{sec3} \label{subsec9.3}}

 It suffices to prove only necessary conditions.

 We suppose set $U$ is  closed. General setup can be reduced easily to this one.

 First we prove Theorem \ref{tqq} if set $U$ is center-symmetric.

 We remind that set $U$ is center-symmetric  if $\thetab \in U$ implies $-\thetab \in U$.

 \begin{lemma} \label{lqq} Suppose that set U is bounded, convex and center-symmetric. Then the statement of Theorem \ref{tqq} holds.
 \end{lemma}

 For any vectors $\thetab_1 \in \mathbb{H}$ and $\thetab_2 \in \mathbb{H}$ define segment $\frak{int}(\thetab_1,\thetab_2) = \{\thetab : \thetab = (1 - \lambda)\,\thetab_1 + \lambda \,\thetab_2, \,\, \lambda \in [0,1]\,\}$.

 Proof of Lemma \ref{lqq} is based on the following Lemma \ref{lqq1}.
 \begin{lemma}\label{lqq1} For any vectors $\thetab_1 \in U$ and $\thetab_2 \in U$ we have $\frak{int}\Bigl(\frac{\thetab_1-\thetab_2}{2},\frac{\thetab_2-\thetab_1}{2}\Bigr) \subset U$.
 There  holds $0 \in \frak{int}\Bigl(\frac{\thetab_1-\thetab_2}{2},\frac{\thetab_2-\thetab_1}{2}\Bigr)$ and segment $ \frak{int}\Bigl(\frac{\thetab_1-\thetab_2}{2},\frac{\thetab_2-\thetab_1}{2}\Bigr)$ is parallel to segment  $\frak{int}(\thetab_1,\thetab_2)$.
 \end{lemma}

 {\sl Remark 3.1.} Let  we have segment $\frak{int}(\thetab_1,\thetab_2) \subset U$.
 Let $\etab$ and $-\etab$ be the points of intersection of the line $L = \{\thetab : \thetab = \lambda(\thetab_1 - \thetab_2), \lambda \in \mathbb{R}^1\}$ and the boundary of set $U$. Then, by Lemma \ref{lqq1}, we get $\|\thetab_1 - \thetab_2\| \le 2\|\etab\|$.

 \begin{proof}[ Proof of Lemma \ref{lqq1}]. Segments $\frak{int}(\thetab_1,\thetab_2) \subset U$ and $\frak{int}(-\thetab_1,-\thetab_2) \subset U$ are parallel.  For each $\lambda \in [0,1]$ we have $ (1-\lambda)\thetab_1 + \lambda \thetab_2 \in \frak{int}(\thetab_1,\thetab_2)$ and $-\lambda \thetab_1 - (1 - \lambda)\thetab_2 \in \frak{int}(-\thetab_1,-\thetab_2)$.
 The middle $\thetab_\lambda= ((1-2\lambda)\,\thetab_1 - (1-2\lambda)\,\thetab_2)/2$ of segment  $\frak{int}((1-\lambda)\thetab_1 + \lambda \thetab_2, -\lambda \thetab_1 - (1 - \lambda)\thetab_2) \subset U$ belongs to segment $\frak{int}\Bigl(\frac{\thetab_1-\thetab_2}{2},\frac{\thetab_2-\thetab_1}{2}\Bigr)$ and, for each point $\thetab$ of segment  $\frak{int}\Bigl(\frac{\thetab_1-\thetab_2}{2},\frac{\thetab_2-\thetab_1}{2}\Bigr)$, there is $\lambda\in [0,1]$ such that $\thetab=\thetab_\lambda$ . Therefore $\frak{int}\Bigl(\frac{\thetab_1-\thetab_2}{2},\frac{\thetab_2-\thetab_1}{2}\Bigr) \subset U$.
 \end{proof}

 \begin{proof}[ Proof of Lemma \ref{lqq}]
   Define sequence of orthogonal vectors $\eb_i$ by induction.

    Define vector $\eb_1$, $\eb_1 \in U$,   such that $\|\eb_1\| = \sup\{ \|\thetab\|, \thetab \in U \}$. Denote $\Pi_1$ linear subspace generated $\eb_1$. Denote $\Gamma_1$ linear subspace orthogonal to $\Pi_1$.

  Define vector $\eb_i \in U \cap \Gamma_{i-1}$  such that $\| \eb_i \|  = \sup\{ \| \thetab\| : \thetab \in U \cap \Gamma_{i-1}\}$.  Denote $\Pi_i$ linear subspace generated vectors $\eb_1, \ldots, \eb_i$. Denote $\Gamma_i$ linear subspace orthogonal to $\Pi_i$.

 Denote $d_i = \| \eb_i \|$. Note that $d_i \to 0$ as $i \to \infty$. Otherwise, by Theorem 5.3 in  \cite{er15}, there does not exist uniformly consistent tests for the problem of testing hypothesis $\mathbb{H}_0 : \thetab = \0b$ versus alternative $\mathbb{H}_1 : \thetab = \eb_i$, $i=1,2, \ldots$.

 For any $\varepsilon \in (0,1)$  denote  $l_\varepsilon = \min \{j\,:\, d_j < \varepsilon, j =1,2, \ldots \}$.

Denote $B_r(\thetab)$ ball having radius $r$ and center $\thetab$.

It suffices to show that, for any $\varepsilon_1 > 0$, there is finite coverage of set $U$ by balls $B_{\varepsilon_1}(\thetab)$.

Denote $\varepsilon = \varepsilon_1/9$.

Denote $U_\varepsilon$ projection of set $U$ onto subspace $\Pi_{l_\varepsilon}$.

Denote $\tilde B_r(\thetab)$ ball in $\Pi_{l_\varepsilon}$ having radius $r$ and center $\thetab \in \Pi_{l_\varepsilon}$. There is ball $\tilde B_{\delta_1}(0)$ such that $\tilde B_{\delta_1}(0) \subset U$. Denote $\delta = \min\{\varepsilon,\delta_1\}$.

Let $\thetab_1, \ldots, \thetab_k$ be $\delta$-net in $U_\varepsilon$.

Let $\etab_1, \ldots, \etab_k$ be points of $U$ such that $\thetab_i$ is projection of $\etab_i$ onto subspace $\Pi_{l_\varepsilon}$ for $1 \le i \le k$.

Let us show that $B_{\varepsilon_1}(\etab_1), \ldots, B_{\varepsilon_1}(\etab_k)$  is coverage of set $U$.

Let $\etab \in U$  and let $\thetab$ be projection of $\etab$ onto $\Pi_{l_\varepsilon}$. There is $i$, $1 \le i \le k$, such that  $\|\thetab_i - \thetab\| \le \delta$. It suffices to show that $\etab \in B_{\varepsilon_1}(\etab_i)$.

By Lemma \ref{lqq1}, $\frak{int}\Bigl(\frac{\etab_i-\etab}{2},\frac{\etab -\etab_i}{2}\Bigr) \subset U$. Since $\thetab_i - \thetab \in \Pi_{l_\epsilon}$ and $\thetab_i - \thetab \in \tilde B_{\delta}(\0b)$, then $(\thetab_i - \thetab)/2 \in U$.  Since set $U$ is center-symmetric and convex we  have  $\frac{1}{2}((\etab_i- \etab)/2) - \frac{1}{2}((\thetab_i - \thetab)/2) \in U$. Note that vector $(\etab_i- \thetab_i) - (\etab - \thetab)$ is orthogonal to the subspace $\Pi_{l_\varepsilon}$. Therefore $\|((\etab_i- \thetab_i) - (\etab - \thetab))/4\| \le 2\varepsilon$. Therefore $\|\etab - \etab_i\| \le 8 \varepsilon + \|\thetab - \thetab_i\| < 9 \varepsilon$. This implies $\etab \in B_{\varepsilon_1}(\etab_i)$.
\end{proof}

\begin{proof}[ Proof of Theorem \ref{tqq}] We say that set $\bar  W$ is trimmed symmetrization of set $W$ if $\xb \in \bar W$ holds, if and only if, $\xb \in W$  and $-\xb \in W$. If $W$ is convex, then $\bar W$ is convex as well.

Since $\bar U \subset U$, then there is consistent tests for problem of testing hypothesis $\thetab = \0b$ versus alternatives $\mathbb{\bar H}_n\,:\, \thetab \in \bar V_n = \{ \thetab\,:\, \|\thetab\|  \ge \rho_n, \thetab \in \bar U\}$ if there is consistent test for sets of alternatives $V_n$.

Therefore set $\bar U$ is compact. We show that this implies that set $U$ is compact as well.

Suppose otherwise. Then there are points $\xb_i \in U$, $1 \le i < \infty$, and positive constant $b$ such that $\rho(\xb_i,\Pi_{i-1}) > b$ for $2 \le i <  \infty$. Here $\Pi_{i-1}$ is hyperplane generated points $\xb_1, \ldots, \xb_{i-1}$.  Then convex hull $L \subset U$ of points $\xb_1, \xb_2, \ldots$ is not compact as well.

Denote $M$  hyperplane generated by points $\xb_i$, $1 \le i < \infty$.
Without loss of generality we can suppose $\0b \notin M$.
There is  $\lambda\, < \, 0$ such that $\xb_0 = \lambda \wb \in U$ where $\wb = \frac{\xb_1 + \xb_2}{2}$.

Denote $K$ convex hull of points $\xb_0, \xb_1, \xb_2, \ldots$. Let $\bar K$ be trimmed symmetrization of $K$. Then $\bar K \subset \bar U$ and therefore $\bar K$ is compact.
Let us show that there is $\delta > 0$ such that $\xb_0 + \delta(\xb_i - \xb_0)  \in \bar K$ for all $i = 1,2,\ldots$. Therefore set of points $\xb_0 + \delta(\xb_i - \xb_0)$,  $i = 1,2,\ldots$, is compact. We will come to contradiction.

Denote $d = \sup\{\, \|\xb - \yb\|\,:\, \xb,\yb \in U \,\}$.

Denote $\alpha_k$ angle between vectors $\xb_k - \xb_0$ and  $\wb - \xb_0$.

Denote $\beta_k$ angle between vectors $\wb - \xb_k$ and $\wb - \xb_0$.

Then angle $\gamma_k$ between vectors $\xb_0 - \xb_0$ and $\wb -\xb_0$ equals $\beta_k - \alpha_k$.

If we show $\gamma_k > c >0$ for all $k$, we prove the existence $\delta > 0$.

Denote $\wb_k$ projection of $\xb_k$ on a line passing through points $\xb_0$ and $\wb$.

Then $\|\xb_k - \wb\| \ge b$ and $\|\xb_0 - \wb_k \| \le \| \xb_0 - \xb_k\| \le d$.

Hence we have
\begin{equation*}
\begin{split} &
\gamma_k  = \arctan \frac{\|\xb_k - \wb_k\|}{\|\wb - \wb_k\|}  - \arctan \frac{\|\xb_k - \wb_k\|}{\|\xb_0 - \wb_k\|}\\&
\\ & \ge 
\arctan \frac{b}{d -\|\xb_0 - \wb\|}  - \arctan \frac{b}{d} > c > 0.
\end{split}
\end{equation*}
\end{proof}

\begin{proof}[ Proof of Theorem \ref{tqq1}]. Proof of Theorem \ref{tqq} is based on Theorem 5.3 in \cite{er15}.  For linear inverse ill-posed problems (\ref{il1}), Theorem 5.5 in \cite{er15} is akin to Theorem 5.3 in \cite{er15}. Thus it suffices to implement Theorem 5.5 in \cite{er15} instead of Theorem 5.3 in \cite{er15} in  proof of Theorem \ref{tqq}.
\end{proof}
 \subsection{Proof of Theorems of section \ref{sec4} \label{subsec9.4}}
\begin{proof}[ Proof of Theorem \ref{tq2}] Theorem \ref{tq2} and its version for Remark \ref{rem1} setup can be deduced straightforwardly from Theorem 1 in  \cite{er90}.

Lower bound follows from  reasoning of Theorem 1 in \cite{er90} straightforwardly.

Upper bound follows from the following reasoning.  We have
\begin{equation}\label{aq2}
\begin{split}&
 \sum_{j=1}^\infty \kappa_{nj}^2 y_j^2  = \sum_{j=1}^\infty \kappa_{nj}^2 \theta_{nj}^2 + 2\frac{\sigma}{\sqrt{n}}\sum_{j=1}^\infty \kappa_{nj}^2 \theta_{nj} \xi_j +
\frac{\sigma^2}{n} \sum_{j=1}^\infty \kappa_{nj}^2 \xi_j^2\\&
= n^{-2} A_n(\thetab_n) + 2\,J_{1n} + J_{2n}
\end{split}
\end{equation}
with
\begin{equation}\label{aq3}
\mathbf{E} [J_{2n}] = \frac{\sigma^2}{n} \rho_n,
\quad
\mathbf{Var} [J_{2n}] = 2 \frac{\sigma^4}{n^4} A_n
\end{equation}
and
\begin{equation}\label{aq5}
\mathbf{Var} [J_{1n}] = \frac{\sigma^2}{n}  \sum_{j=1}^\infty \kappa_{nj}^4 \theta_{nj}^2 \le \frac{\sigma^2\kappa^2_n}{n}  \sum_{j=1}^\infty \kappa_{nj}^2 \theta_{nj}^2= o(n^{-4}A_n(\thetab_n)).
\end{equation}
By Chebyshov inequality, it follows from (\ref{aq2}) - (\ref{aq5}), that, if $A_n  = o(A_n(\thetab_n))$ as $ n \to \infty$, then $\beta(L_n,f_n) \to 0$ as $n \to \infty$. Thus it suffices to explore the case
\begin{equation}\label{aq6}
A_n\asymp A_n(\thetab_n) = n^2\sum_{j=1}^\infty \kappa_{nj}^2 \theta_{nj}^2.
\end{equation}
If (\ref{aq6}) holds, then, implementing the reasoning of proof of Lemma 1 in \cite{er90}, we get that (\ref{aq1}) holds. \end{proof}
\begin{proof}[ Proof of Theorem  \ref{tq3}]  Let (\ref{con2}) hold. Then, by A5 and (\ref{u1}), we have
\begin{equation*}
A_n(\thetab_n) = n^2 \sum_{j=1}^\infty \kappa_{nj}^2 \theta_{nj}^2 \ge C n^2\kappa^2_n \sum_{j=1}^{c_2k_n}  \theta_{nj}^2 \asymp n^2 \kappa_n^2 n^{-2r}  \asymp 1.
\end{equation*}
By Theorem \ref{tq2}, this implies sufficiency.

Necessary conditions follows from sufficiency conditions in Theorem \ref{tq4}. \end{proof}
\begin{proof}[ Proof of Theorem  \ref{tq4}] Let (\ref{con3}) hold. Then, by (\ref{u1}) and A2, we have
\begin{equation}\label{eq2}
A_n(\thetab_n) \le C n^2\kappa^2_n \sum_{j < c_2k_n}  \theta_{nj}^2 + C n^2 \kappa^2_{n,[c_2n]} \sum_{j > c_2 n} \theta_{nj}^2 \asymp o(1) +  O(\kappa^2_{n,[c_2n]}/\kappa^2_n).
\end{equation}

 By A4, we have
\begin{equation}\label{eq102}
\lim_{c_2 \to \infty}\lim_{n \to \infty} \kappa^2_{n,[c_2n]}/\kappa^2_n \to 0,
\end{equation}

By Theorem \ref{tq2}, (\ref{eq2}) and (\ref{eq102}) together, we get sufficiency.
 \end{proof}

\begin{proof}[ Proof of Theorem  \ref{tq1}] Statement {\it i.} follows from  Theorem \ref{tq3} and Lemma \ref{ld1} provided below.

\begin{lemma} \label{ld1} Let $f_n \in \mathbb{\bar B}^s_{2\infty}(c_1)$ and $cn^{-r}\le \|f_n\| \le Cn^{-r}
$. Then, for $l_n = C_1 n^{2 - 4r}(1  +o(1)) = C_1 n^{\frac{r}{s}}(1+o(1)) $ with $C_1^{2s} > 2c_1/c$, there holds
\begin{equation}\label{d1}
\sum_{j=1}^{l_n} \theta_{nj}^2 > \frac{c}{2} n^{-2r}(1 + o(1)).
\end{equation}
\end{lemma}
Proof. Let $f_n \in c_1 U$. Then we have
\begin{equation*}
l_n^{2s} \sum_{j=l_n}^\infty \theta_{nj}^2 = C_1^{2s} n^{2r}  \sum_{j=l_n}^\infty \theta_{nj}^2(1 + o(1))\le c_1(1 + o(1)).
\end{equation*}
Hence
\begin{equation} \label{ucc1}
\sum_{j=l_n}^\infty \theta_{nj}^2 \le c_1 C_1^{-2s} n^{-2r} \le \frac{c}{2} n^{-2r}(1 + o(1)).
\end{equation}
Therefore (\ref{d1}) holds. \end{proof}

\begin{proof}[ Proof of Theorem  \ref{tq1} ]  Suppose  opposite that {\it ii.} does not valid.
Then  $f = \sum_{j=1}^\infty \tau_{j}\,\phi_j  \notin \mathbb{\bar B}^s_{2\infty}$. This implies that there is sequence $m_l$, $m_l \to \infty$ as $l \to \infty$, such that
\begin{equation}\label{u5}
m_l^{2s} \sum_{j=m_l}^\infty \tau_{j}^2 = C_l
\end{equation}
with $C_l \to \infty$ as $l \to \infty$.

Define a sequence $\etab_l = \{\eta_{lj}\}_{j=1}^\infty$ such that
$\eta_{lj} = 0$ if $j  < m_l$ and $\eta_{lj} = \tau_j$ if $j \ge m_l$.

Since $V$ is convex and ortho-symmetric we have  $\tilde f_l=  \sum_{j=1}^\infty \eta_{lj}\,\phi_j \in V$.

For alternatives $\tilde f_l$ we define sequence $ n_l$ such that
\begin{equation}\label{u5b}
\|\etab_l\|^2 \asymp n_l^{-2r}\asymp m_l^{-2s} C_l .
\end{equation}
Then
\begin{equation}\label{u7}
n_l \asymp C_l^{-1/(2r)} m_l^{s/r} = C_l^{-1/(2r)} m_l^{\frac{1}{2 - 4r}}.
\end{equation}
Therefore we get
\begin{equation}\label{u10}
m_l \asymp C_l^{(1-2r)/r}n_l^{2-4r}.
\end{equation}
By A4, (\ref{u10}) implies
\begin{equation}\label{u9}
\kappa^2_{n_l m_l} = o(\kappa^2_{n_l}).
\end{equation}
Using (\ref{u1}), A2 and (\ref{u9}), we get
\begin{equation}\label{u12}
\begin{split}&
A_{n_l}(\etab_l) = n_l^2 \sum_{j=1}^{\infty} \kappa_{n_lj}^2\eta_{jl}^2 \le n_l^{2} \kappa^2_{m_l n_l}\sum_{j=m_l}^{\infty} \theta_{n_lj}^2\\& \asymp n_l^{2-2r}\kappa^2_{n_l m_l} = O(\kappa^2_{n_l m_l }\kappa^{-2}_{n_l}) = o(1) .
\end{split}
\end{equation}
By  Theorem \ref{tq2},  (\ref{u12}) implies $n_l^{-r}$-inconsistency of   sequence of alternatives $\tilde f_l$. \end{proof}
\begin{proof}[ Proof of Theorem \ref{tq7}]  Theorem \ref{tq7} follows  from Lemmas \ref{ld3} -- \ref{ld6}.
\begin{lemma}\label{ld3} For any  $c$ and any $C$ there is $\mathbb{\bar B}^s_{2\infty}(P_0)$ such that, if $
f_n = \sum_{j =1}^{ck_n} \theta_{nj} \phi_j,
$ and $\| f_n \| \le C n^{-r}$,
then  $f_n \in \mathbb{\bar B}^s_{2\infty}(P_0)$.
\end{lemma}
\begin{proof} Let $C_1$ be such that $k_n = C_1 n^{r/s}(1 + o(1))$. Then we have
\begin{equation*}
k_n^{2s} \sum_{j=1}^{ck_n} \theta_{nj}^2 \le  C_1 n^{2r}  \sum_{j=1}^{\infty} \theta_{nj}^2  (1 + o(1)) < C C_1 (1 + o(1)).
\end{equation*}
\end{proof}
\begin{lemma}\label{ld4} Necessary conditions  in Theorem \ref{tq7} are fulfilled.
\end{lemma}
\begin{proof} Let $f_n = \sum_{j=1}^\infty \theta_{nj}  \phi_j$ and let   $f_{1n} = \sum_{j=1}^\infty \eta_{nj}  \phi_j$. Denote $  \zeta_{nj}  = \theta_{nj} - \eta_{nj}$, $1 \le j < \infty$.

For any $\delta > 0$,  $c_1$ and $C_2$, there is $c_2$ such that, for each sequence $f_{1n} \in \mathbb{\bar B}^s_{2\infty}(P_0)$, $\| f_{1n} \| \le C_2 n^{-r}$, there holds
\begin{equation}\label{uh11}
\sum_{j>c_2 k_n} \eta_{nj}^2 < \delta n^{-2r}.
\end{equation}
To prove (\ref{uh11}) it suffices to put $c_2 k_n = l_n = C_1 n^{2-4r}(1 + o(1))$ in (\ref{ucc1}) with $C_1^{2s} > \delta c_1$.

We have
\begin{equation}\label{dub2}
\begin{split}&
J_n=\left| \sum_{j > ck_n} \theta_{nj}^2 - \sum_{j > ck_n} \zeta_{nj}^2\right| \le  \sum_{j > ck_n} |\eta_{nj}(2\theta_{nj} - \eta_{jn})|\\& \le \left( \sum_{j > ck_n}\eta_{nj}^2\right)^{1/2}\left(2\left( \sum_{j > ck_n} \theta_{nj}^2\right)^{1/2} + \left( \sum_{j > ck_n}\eta_{nj}^2\right)^{1/2}\right)
\le  C\delta^{1/2}n^{-2r}.
\end{split}
\end{equation}
By (\ref{ma1}), using (\ref{uh11}) and (\ref{dub2}), we get
\begin{equation}\label{dub3}
\begin{split}&
 \sum_{j < ck_n} \theta_{nj}^2=
 \sum_{j=1}^\infty \eta_{nj}^2  + \sum_{j=1}^\infty \zeta_{nj}^2 - \sum_{j \ge ck_n} \theta_{nj}^2 \ge
  \sum_{j < ck_n} \eta_{nj}^2  -J_n\\&
 \ge  \sum_{j < ck_n} \eta_{nj}^2   - C\delta^{1/2}n^{-2r}
 \ge \|f_{1n}\|^2 - \delta n^{-2r} - C\delta^{1/2}n^{-2r}
 .
 \end{split}
\end{equation}
By Theorem \ref{tq3}, (\ref{dub3}) implies consistency of sequence $f_n$. \end{proof}

\begin{lemma}\label{ld6} Let sequence of alternatives $f_{n}$, $cn^{-r}\le \|f_{n}\| \le Cn^{-r}
$, be consistent. Then (\ref{ma1}) holds.
\end{lemma}
\begin{proof} By Theorem \ref{tq3}, there  are $c_1$ and $c_2$  such that sequence $f_{1n} = \sum_{j<c_2 k_n} \theta_{nj}  \phi_j$ is consistent and $\|f_{1n}\| \ge c_1 n^{-r}$.
By Lemma \ref{ld3}, there is
$\mathbb{\bar B}^s_{2\infty}(P_0)$ such that $f_{1n} \in \mathbb{\bar B}^s_{2\infty}(P_0)$.\end{proof} \end{proof}

\begin{proof}[ Proof of Theorem \ref{tq11}] By A4 and (\ref{u1}), for any $\delta > 0$, there is $c$ such that we have
\begin{equation}\label{dub29}
n^2 \sum_{j > ck_n} \kappa_{nj}^2 \theta_{nj}^2 \le \delta.
\end{equation}
By Lemma \ref{ld3}, there is $P_0$ such that $f_{1n} = \sum_{j <ck_n} \theta_{nj} \phi_j \in \mathbb{\bar B}^s_{2\infty}(P_0)$. By Theorem \ref{tq2} and (\ref{dub29}), for sequence of alternatives $f_{1n}$,
(\ref{uuu}) and (\ref{uu1}) hold. \end{proof}
\begin{proof}[ Proof of Theorem \ref{tq5}]  Let $f_n = \sum_{j=1}^\infty \theta_{nj} \phi_j$  and let  $f_{1n} = \sum_{j=1}^\infty \eta_{nj} \phi_j$. Denote $\etab_n = \{\eta_{nj}\}_{j=1}^\infty$.

By Cauchy inequality, we have
\begin{equation} \label{co1}
\begin{split}&
|A_n(\thetab_n) - A_n(\thetab_n + \etab_n)| = n^2\Bigl| \sum_{j=1}^\infty \kappa_{nj}^2 \theta_{nj}^2  - \sum_{j=1}^\infty \kappa_{nj}^2 (\theta_{nj} + \eta_{nj})^2\Bigr|\\& \le
2\,A_n^{1/2}(\thetab_n)A_n^{1/2}(\etab_n) + A_n(\etab_n) .
\end{split}
\end{equation}
By Theorem \ref{tq2}, inconsistency of sequence $f_{1n}$ implies $A_n(\etab_n) = o(1)$ as $n \to \infty$. Therefore, by (\ref{co1}),  $|A_n(\thetab_n) - A_n(\thetab_n + \etab_n)| = o(1)$ as $n \to \infty$. Hence, by Theorem \ref{tq2}, we get Theorem \ref{tq5}. \end{proof}

\begin{proof}[ Proof of Theorem \ref{tq6} ] For proof of sufficiency suppose  opposite. Then there is sequence $n_i$, $n_i \to \infty$ as $i \to \infty$ such that $f_{n_i} = f_{1n_i} + f_{2n_i}$,
\begin{equation}\label{dub401}\|f_{n_i}\|^2 = \|f_{1n_i}\|^2 + \|f_{2n_i}\|^2,
\end{equation}
 $c_1 n_i^{-r}<\|f_{1n_i}\| < C_1 n_i^{-r}$, $c_2 n_i^{-r}<\|f_{2n_i}\| < C_2 n_i^{-r}$ and sequence $f_{2n_i}$ is inconsistent.
 \vskip 0.1cm
 Let $f_{n_i} = \sum_{j=1}^\infty \theta_{n_ij}\phi_j$, $f_{1n_i} = \sum_{j=1}^\infty \theta_{1n_ij}\phi_j$ and $f_{2n_i} = \sum_{j=1}^\infty \theta_{2n_ij}\phi_j$.

  Then, by Theorem  \ref{tq4} and by (\ref{con19}),  we get that there are $\varepsilon_i$,  $\varepsilon_i\to 0$ and $C_i = C(\varepsilon_i)$,  $C_i \to \infty$ as $i \to \infty$ such that
\begin{equation}\label{dub301}
\sum_{j > C_ik_n} \theta_{n_ij}^2 =  \sum_{j > C_ik_n} (\theta_{1n_ij} + \theta_{2n_ij})^2 = o(n^{-2r}), \quad
 \sum_{j < C_ik_n} \theta_{2n_ij}^2  = o(n^{-2r}).
\end{equation}
By (\ref{dub401}) and (\ref{dub301}), we get
\begin{equation}\label{dub402}
\sum_{j = 1}^\infty \theta_{n_ij}^2 = \sum_{j < C_ik_n} \theta_{n_ij}^2 + o(n^{-2r}) = \sum_{j < C_ik_n} \theta_{1n_ij}^2 + o(n^{-2r}).
\end{equation}
Hence, by (\ref{dub401}), we get $\|f_{2n_i}\| = o(n^{-r})$. We come to contradiction.

To prove  necessary conditions suppose (\ref{con19}) does not hold. Then there are $\varepsilon > 0$ and sequences $C_i \to \infty$, $n_i \to \infty$ as $i \to \infty$ such that
\begin{equation*}
\sum_{j > C_ik_{n_i}} \theta_{n_ij}^2 > \varepsilon n_i^{-2r} .
\end{equation*}
Then, by A4 and (\ref{u1}), we get
\begin{equation*}
n_i^2\sum_{j > C_ik_{n_i}} \kappa^2_{n_ij}\theta_{n_ij}^2 =o(1).
\end{equation*}
Therefore, by Theorem \ref{tq2}, subsequence $f_{1n_i} = \sum_{j > C_ik_{n_i}} \theta_{n_ij} \phi_j$ is inconsistent. \end{proof}
\begin{proof}[ Proof of Theorem \ref{tq12}] For proof of necessary conditions, it suffices to put $$f_{1n} = \sum_{j < C_1(\epsilon)k_n} \theta_{nj}\phi_j.$$  By Lemma \ref{ld3}, there is $P_0 > 0$ such that $f_{1n} \in \mathbb{\bar B}^s_{2\infty}(P_0)$.
Proof of sufficiency  is simple and is omitted. \end{proof}

\begin{proof}[ Proof of Theorem \ref{tq8}] Necessary conditions are rather evident, and proof is omitted. Proof of sufficiency is also simple.
\begin{lemma}\label{ld5} Let for sequence $f_n$, $c n^{-r}<\| f_n \| < C n^{-r}$, (\ref{ma2}) hold. Then sequence $f_n$ is purely $n^{-r}$-consistent.
\end{lemma}
Suppose $f_n = \sum_{j = 1}^\infty \theta_{nj} \phi_j$  is not purely $n^{-r}$-consistent.
Then, by Theorem \ref{tq6}, there are $c_1$ and sequences $n_i$,  and $c_{n_i}$, $c_{n_i} \to \infty$ as $i \to \infty$, such that
\begin{equation*}
\sum_{j > c_{n_i} k_{n_i}} \theta_{n_lj}^2 >c_1 n_i^{-r}.
\end{equation*}
Therefore, if we put $f_{1n_i} = \sum_{j > c_{n_i} k_{n_i}} \theta_{n_ij} \phi_j $, then (\ref{ma2}) does not hold. \end{proof}
\subsection{Proof of Theorems of section \ref{sec5} \label{subsec9.5}}
\begin{proof}[ Proof of  version of Theorem \ref{tq3}. ] Since $\hat K(\omega)$ is analytical function and $\hat K(0)=1$ there is $b> 0$ such that $|\hat K(\omega)| > c > 0$ for $|\omega| < b$.

Let (\ref{con2}) hold. Then we have
\begin{equation*}
\begin{split}&
T_{1n}(f_n)  = \sum_{j=-\infty}^\infty |\hat K(jh_n)|^2 |\theta_{nj}|^2 \ge \sum_{|j| h_n < b} |\hat K(jh_n)|^2 |\theta_{nj}|^2\\& \asymp
\sum_{|j| < c_2k_n} |\hat K(jh_n)|^2 |\theta_{nj}|^2
 \asymp  n^{-1}h_n^{-1/2} \asymp n^{-2r}
 \end{split}
\end{equation*}
for $c_2k_n < b h_n^{-1}$. By Theorem \ref{tk2}, this implies consistency.\end{proof}

\begin{proof}[ Proof of  version of Theorem \ref{tq1}] We verify only {\it iv.}.
Let $f = \sum_{j=-\infty}^\infty \tau_j \phi_j \notin \mathbb{B}^s_{2\infty}$.  Then there is  sequence $m_l$, $m_l \to \infty$ as $l \to \infty$, such that
\begin{equation}\label{bb5}
m_l^{2s} \sum_{|j|\ge m_l}^\infty |\tau_j|^2 = C_l
\end{equation}
with $C_l \to \infty$ as $l \to \infty$.

It is clear that  we can define a sequence $m_l$ such that
\begin{equation}\label{gqq}
m_l^{2s} \sum_{m_l \le |j| \le 2m_{l}} |\tau_j|^2 > \delta C_l,
\end{equation}
where $\delta$, $0< \delta <1/2$,  does not depend on $l$.

Otherwise, we have
 \begin{equation*}
 2^{2s(i-1)} m_l^{2s} \sum_{j=2^{i-1}m_l}^{2^{i}m_l}  \tau_j^2  < \delta  C_l
\end{equation*}
 for  all $i = 1,2,\ldots$,  that implies that the left hand-side of
 (\ref{bb5}) does not exceed $2\delta C_l$.

Define a sequence $\etab_l = \{\eta_{lj}\}_{j=-\infty}^\infty$ such that
 $\eta_{lj} = \tau_j$ if  $|j| \ge m_{l}$, and $\eta_{lj} = 0$ otherwise.

Denote
$$
\tilde f_l(x) =  \sum_{j=-\infty}^\infty \eta_{lj} \exp\{2\pi ijx\}.
$$
 For alternatives $\tilde f_l(x)$ we define sequence $n_l$ such that $\|\tilde f_l(x)\| \asymp n_l^{-r}$.

 Then
\begin{equation*}
n_l \asymp C_l^{-1/(2r)} m_l^{s/r}.
\end{equation*}
We have $|\hat K(\omega)| \le \hat K(0) = 1$ for all $\omega \in R^1$ and $|\hat K(\omega)| > c > 0$ for all $ |\omega| < b$. Hence, if we put $h_l= h_{n_l} =2^{-1}b^{-1}m_l^{-1}$, then, by (\ref{gqq}), there is $C > 0$  such that, for all $h> 0$, there holds
\begin{equation*}
T_{1n_l}(\tilde f_l,h_l) = \sum_{j=-\infty}^\infty |\hat K(jh_l)\,\eta_{lj}|^2 > C \sum_{j=-\infty}^\infty |\hat K(jh)\,\eta_{lj}|^2 = C T_{1n_l}(\tilde f_l,h).
\end{equation*}
Thus we can choose $h = h_l$ for further reasoning.

By (\ref{gqq}), we get
\begin{equation}\label{k101}
T_{1n_l}(\tilde f_l) =  \sum_{|j|>m_l}\, |\hat K(jh_l) \eta_{lj}|^2  \asymp \sum_{j=m_l}^{2m_l} |\eta_{lj}|^2 \asymp n_l^{-2r}.
\end{equation}
If we put in estimates (\ref{u7}),(\ref{u10}), $k_l = [h_{n_l}^{-1}]$ and $m_l = k_l$, then we get
\begin{equation}\label{k102}
h_{n_l}^{1/2} \asymp C_l^{(2r-1)/2}n_l^{2r-1}.
\end{equation}
By (\ref{k101}) and (\ref{k102}), we get
\begin{equation*}
n_l T_{1n_l}(\tilde f_l)h_{n_l}^{1/2} \asymp C_l^{-(1-2r)/2}.
\end{equation*}
By Theorem \ref{tk2}, this implies inconsistency of sequence of alternatives $\tilde f_l$. \end{proof}
\subsection{Proof of Theorems of section \ref{sec6} \label{subsec9.6}}
  We have
\begin{equation*}
n^{-1}m_n^{-1}T_n(F) = \sum_{l=0}^{m_n-1}\left(\int_{l/m_n}^{(l+1)/m_n} f(x) dx \right)^2.
\end{equation*}
Using representation  $f(x)$ as Fourier series
\begin{equation*}
f(x) = \sum_{j=-\infty}^\infty \theta_j \exp\{2\pi ijx\},
\end{equation*}
we get
\begin{equation*}
\int_{l/m_n}^{(l+1)/m_n} f(x) dx = \sum_{j=-\infty}^\infty \frac{\theta_j}{2\pi ij}\exp\{2\pi ijl/m_n\} ( \exp\{2\pi ij/m_n\} - 1)
\end{equation*}
for $ 1 \le l < m_n$.

In what follows, we shall use  the following agreement $0/0= 0$.
\begin{lemma}\label{ch1} There holds
\begin{equation}\label{ach1}
n^{-1}m_n^{-1}T_n(F) = m_n \sum_{k=-\infty}^\infty \sum_{j \ne km_n} \frac{\theta_j \bar\theta_{j-km_n}}{4\pi^2j(j-km_n)}(2 - 2 \cos(2\pi j/m_n)).
\end{equation}
\end{lemma}
\begin{proof}[Proof of Lemma \ref{ch1}] We have
\begin{equation}\label{uh105}
\begin{split}&
n^{-1}m_n^{-1}T_n(F) =\sum_{l=0}^{m_n-1}\Bigl(\sum_{j \ne 0}\frac{\theta_j}{2\pi ij}\exp\{2\pi ijl/m_n\} ( \exp\{2\pi ij/m_n\} - 1)\Bigr)\\& \times
\Bigl(\sum_{j \ne 0} \frac{-\bar\theta_j}{2\pi ij}\exp\{-2\pi ijl/m_n\} (\exp\{-2\pi ij/m_n\} - 1)\Bigr) = J_1 + J_2
\end{split}
\end{equation}
with
\begin{equation}\label{uh106}
\begin{split}&
J_1=\sum_{l=0}^{m_n-1}\sum_{k=-\infty}^\infty\,\,\, \sum_{j_1 = j -km_n} \frac{\theta_j \bar\theta_{j_1}}{4\pi^2jj_1}\exp\{2\pi ilk\}\\&\times(\exp\{2\pi ij/m_n\} - 1)(\exp\{-2\pi ij_1/m_n\} - 1)\\&= m_n \sum_{k=-\infty}^\infty \sum_{j=-\infty}^\infty \frac{\theta_j \bar\theta_{j-k m_n}}{4\pi^2j(j-k m_n)}(2 - 2 \cos(2\pi j/m_n))
\end{split}
\end{equation}
and
\begin{equation}\label{ux}
\begin{split}&
J_2=\sum_{l=0}^{m_n-1}\sum_{j \ne 0} \sum_{j_1 \ne j-k m_n} \frac{\theta_j \bar\theta_{j_1}}{4\pi^2 j j_1} \exp\{2\pi i(j - j_1)l/m_n\}\\&\times(\exp\{2\pi ij/m_n\} - 1)(\exp\{-2\pi ij_1/m_n\} - 1) =0,
\end{split}
\end{equation}
where $j_1 \ne j-k m_n$ signifies that summation is performed over all $j_1$ such that $j_1 \ne j-k m_n$ for all integer $k$.

In the last equality of  (\ref{ux}), we make use of the identity
\begin{equation*}
 \sum_{l=0}^{m_n-1}\exp\{2\pi i(j - j_1)l/m_n\} = \frac{\exp\{2\pi i(j - j_1) m_n/m_n\} -1}{\exp\{2\pi i(j - j_1)/m_n\} -1} =0,
\end{equation*}
if $j-j_1 \ne k m_n$ for all integer $k$.

By (\ref{uh105}) - (\ref{ux}) together, we get (\ref{ach1}). \end{proof}

For any c.d.f $F$ and any $k$ denote $\tilde F_k$ the function having the derivative
$$
 1 + \tilde f_k(x)  = 1 + \sum_{|j| > k} \theta_j \exp\{2\pi ijx\}
$$
and such that $\tilde F_k(1) = 1$.

 Denote $i_n = [d m_n]$ where $d > 1+c$.
\begin{lemma}\label{lch2} There holds
\begin{equation}\label{hu}
n^{-1}m_{n}^{-2}T_{n}(\tilde F_{i_n})  \le C m_n^{-1} i_n^{-1}\sum_{|j| > i_n} |\theta_j|^2.
\end{equation}
\end{lemma}
\begin{proof} Denote $\eta_j = \theta_j$ if $|j|> i_n$ and $\eta_j =0$ if $|j| < i_n$.

We have
\begin{equation*}
\begin{split}&
n^{-1}m_{n}^{-2}T_{n}(\tilde F_{i_n}) =\sum_{k=-\infty}^\infty \,\,\sum_{j \ne k m_n}
\frac{\eta_j \bar\eta_{j-k m_n}}{4\pi^2j(j-k m_n)}(2 -2 \cos(2\pi j/m_n))\\&
\le C  \sum_{|j|> i_n} \Bigl|\frac{\eta_j}{j}\Bigr| \sum_{k=-\infty}^\infty  \Bigl|\frac{\eta_{j+k m_n}}{j+k m_n}\Bigr|
\\&
= C\sum_{j=1}^{m_n} \sum_{k=-\infty}^\infty \Bigl|\frac{\eta_{j+k m_n}}{j+k m_n}\Bigr|  \sum_{k_1=-\infty}^\infty
\Bigl|\frac{\eta_{j+(k+k_1)m_n}}{j+(k+k_1)m_n}\Bigr| \\&
 = C\sum_{j=1}^{m_n} \Bigl(\sum_{k=-\infty}^\infty \Bigl|\frac{\eta_{j+k m_n}}{j+k m_n}\Bigr|\Bigr)^2\\&
\le  C\sum_{j=1}^{m_n} \Bigl(\sum_{|k|> d-1} |\eta_{j+k m_n}|^2\Bigr)
  \Bigl(\sum_{|k|> d-1} (j+k m_n)^{-2}\Bigr)
\\& \le C \sum_{j=-\infty}^\infty  |\eta_j|^2 \sum_{|k| > d}(k m_n)^{-2}
 \le C m_n^{-1} i_n^{-1}\sum_{|j| > i_n} |\theta_j|^2.
\end{split}
\end{equation*}
\end{proof}
\begin{proof}[ Proof of version of Theorem \ref{tq3}.] We prove sufficiency. Suppose (\ref{con2}) holds.
Denote
$$
\tilde f_{n}  = \tilde f_{n,c_2k_n}  =\sum_{|j| > c_2 k_n} \theta_{nj}\phi_j\quad\mbox{and}\quad \bar f_{n} = \bar f_{n,c_2k_n} = f_n - \tilde f_{n}
$$
Denote $\tilde F_{n}$, $\bar F_{n}$  the functions having derivatives $1+\tilde f_{n,c_2k_n}$ and $1+\bar f_{n,c_2k_n}$ respectively and such that $\tilde F_{n}(1) = 1$   and $\bar F_{n}(1) =1$.

Let $T_n$ be  chi-squared test statistics with a number of cells $m_n= [c_3 k_n]$  where $c_2 < c_3 $.
Denote $\mathbb{L}_{2,n}$ linear space generated functions ${\bf 1}_{\{x \in ((j-1)/m_n,j/m_n)\}}$, $1  \le j \le m_n$.

Denote $\bar h_n$  orthogonal projection of $\bar f_n$  onto $\mathbb{L}_{2,n}$. Denote $\tilde h_n$ orthogonal projection of  $\tilde f_n$ onto the line $\{h\,:\, h = \lambda \bar h_n,\, \lambda \in \mathbb{R}^1\}$.

Note that $n^{-1/2} T_n^{1/2}(F_n)$ equals the $\mathbb{L}_{2,n}$-norm of function $f_{n}$. Hence we have
\begin{equation}\label{ee9}
n^{-1/2}\,m_n^{-1} T_n^{1/2}(F_n) \ge  \|\bar h_n + \tilde h_n\|.
\end{equation}
Thus, by Theorem \ref{chi2}, it suffices to show that, for some choice of $c_3$, there holds $\|\bar h_n + \tilde h_n\| \asymp n^{-r}$ if $m_n > c_3\, k_n$.

Denote $\bar g_n = \bar f_n - \bar h_n$ and $\tilde g_n = \tilde f_n - \tilde h_n$.

Denote
$$
\bar p_{jn} = \frac{1}{m_n} \int_{(j-1)/m_n}^{j/m_n} \bar f_n(x) dx, \quad 1 \le j \le m_n.
$$
By Lemmas 3 and 4 in section 7 of \cite{ul}, we have
\begin{equation}\label{h5}
 \|\bar g_n\|^2 = m_n\sum_{j=1}^{m_n} \int_{(j-1)/m_n}^{j/m_n}(\bar f_n(x) - \bar p_{jn})^2 \, dx \le 2 \omega^2\Bigl(\frac{1}{m_n}, \bar f_n\Bigr).
\end{equation}
Here
$$\omega^2(h,f) = \int_0^1 (f(t+h) - f(t))^2\,dt, \quad h>0,$$
for any $f \in \mathbb{L}_2^{per}$.
If $f = \sum_{j=-\infty}^\infty \theta_j \phi_j,
$ then
\begin{equation}\label{h6}
\omega^2(s,f) =  2\sum_{j=1}^\infty |\theta_j|^2\,(2  - 2\cos (2\pi js)).
\end{equation}
 Since $1 - cos( x) \le x^2$, then, by (\ref{h5}) and (\ref{h6}), we have
\begin{equation}\label{h8}
\|\bar g_n\|  \le\, 4 \pi (c_2\,k_n/m_n )^{1/2}\, \|\bar f_n\|= \delta \|\bar f_n\|(1+ o(1)),
\end{equation}
where $\delta = 4\pi\,(c_2/c_3)^{1/2}$.

By (\ref{con2}), (\ref{h5}) and (\ref{h8}), we get that there is $c_{30}$, such that
\begin{equation}\label{h81}
\| \bar h_n \| > \frac{c_1}{2} n^{-r}
\end{equation}
for $c_3 > c_{30}$.

For any functions $g_1, g_2 \in \mathbb{L}_2(0,1)$ denote $(g_1,g_2)$ inner product of $g_1$ and $g_2$.

We have
\begin{equation}\label{ee10}
0 = (\bar f_n, \tilde f_n) = (\bar h_n, \tilde h_n) + (\bar g_n, \tilde f_n).
\end{equation}
By (\ref{h8}), we get
\begin{equation*}
|(\bar g_n, \tilde f_n)| \le \|\bar g_n\|\, \|\tilde f_n\| \le \delta C^2 n^{-2r}.
\end{equation*}
Therefore we get
\begin{equation}\label{ee11}
|(\bar h_n, \tilde h_n) | \le \delta C^2 n^{-2r}.
\end{equation}
By (\ref{h81}) è (\ref{ee11}), we get that, for sufficiently small $\delta> 0$, there holds $\|\bar h_n + \tilde h_n\| \asymp n^{-r}$. Hence, using (\ref{ee9}) and implementing Theorem  \ref{chi2}, we  get sufficiency.\end{proof}
\begin{proof}[ Proof of version of Theorem \ref{tq4}] We prove sufficiency. Let $k_n= [c_1 n^{2 - 4r}]$.
For $c_2 > 2c_1$, we have
\begin{equation}\label{h11}
 T_n^{1/2}(F_{n}) \le T_n^{1/2}(\bar F_{n}) + T_n^{1/2}(\tilde F_{n}).
\end{equation}
By Lemma \ref{lch2}, we have
\begin{equation}\label{h12}
n^{-1} m_n{-2} T_n(\tilde F_{n}) \le c_2^{-1}m_n k_n^{-1} \|\tilde f_n \|^2  \le c_2^{-1} c_1 C n^{-2r}.
\end{equation}
We have
\begin{equation}\label{h13}\|\bar f_n\| \ge n^{-1/2}m_n^{-1}T_n^{1/2}(\bar F_n).
\end{equation}
Since one can take arbitrary value $c_2$, $c_2 > 2c_1$, then, by Theorem \ref{chi2}, (\ref{con3}) and  (\ref{h11}) - (\ref{h13}) together, we get
inconsistency of sequence $f_n$. \end{proof}

\begin{proof}[ Proof of  version of Theorem \ref{tq1}] Let us prove {\it ii.}
 Suppose  opposite. Then there is sequence $i_l$, $i_l \to \infty$ as $l \to \infty$, such that
\begin{equation*}
i_l^{2s} \|\tilde f_{i_l}\|^2  = C_l,
\end{equation*}
with $C_l \to \infty$ as $l \to \infty$. Here $f = \sum_{j=-\infty}^\infty\tau_j \phi_j$ and $\tilde f_{i_l} = \sum_{|j| >i_l}\tau_j \phi_j$.

Define  sequence $n_l$  such that $n_l^{-r} \asymp \|\tilde f_{i_l}\|$ as $l \to \infty$.

Then, estimating similarly to  (\ref{u7}) and (\ref{u10}), we get $i_{l}^{-1/2} \asymp C_l^{(2r-1)/2}n_l^{2r-1}$ as $l \to \infty$.

If $m_l = o(i_l)$, then, by Lemma \ref{lch2}, we get
\begin{equation}\label{xy1}
m_{l}^{-1/2}T_{n_l}(\tilde F_{i_l}) \le m_{l}^{1/2} i_l^{-1} n_l \sum_{|j| > i_l} |\tau_j|^2 \asymp m_{l}^{1/2} i_l^{-1} n_l^{1-2r} = o( C_l^{(2r-1)/2}).
\end{equation}

Let $m_l \asymp i_l$ or $i_l = o(m_l)$. Then we have
\begin{equation}\label{udav1}
n_l^{-2r} \asymp \|\tilde f_{i_l}\|^2  \ge n_l^{-1}m_l^{-2} T_{n_l}(\tilde F_{i_l}).
\end{equation}
Therefore
\begin{equation}\label{xy2}
m_l^{-1/2} T_{n_l}(\tilde F_{i_l}) \le Cm_l^{-1/2}n_l^{1-2r} = C m_l^{-1/2}i_l^{1/2} C_l^{(2r-1)/2}  = o(1).
\end{equation}
By  Theorem \ref{chi2}, (\ref{xy1}) -(\ref{xy2}) imply {\sl ii.} \end{proof}
\begin{proof}[ Proof of version of Theorem \ref{tq11}] Let
$
f_{1n} = \sum_{|j|< ck_n} \theta_{nj} \phi_j.
$
Then, by Lemma \ref{ld3}, there is maxiset $\mathbb{\tilde B}^s_{2\infty}(P_0)$ such that $f_{1n} \in \mathbb{\tilde B}^s_{2\infty}(P_0)$.

Denote $F_{1n}$ function having derivative $1 + f_{1n}$ and such that $F_{1n}(1) = 1$.

We have
\begin{equation}\label{uxa1}
|T_n^{1/2}(F_n) - T_n^{1/2}(F_{1n})| \le T_n^{1/2}(F_n - F_{1n} + F_0).
\end{equation}
If $m_n = [c_0 k_n]$  and $c > 2c_0$, then, by Lemma \ref{lch2}, we have
\begin{equation}\label{uxa2}
n^{-1} T_n(F_n - F_{1n} + F_0)  \le c_0 c^{-1}\, \| f_n - f_{1n}\|^2.
\end{equation}
Since the choice of $c$ is arbitrary, by Theorem \ref{chi2}, (\ref{uxa1}) and (\ref{uxa2}) imply (\ref{uuu}) and (\ref{uu1}). \end{proof}

  Proof of  {\sl i.} in version of Theorem \ref{tq1} and versions of Theorems \ref{tq7}, \ref{tq6}, \ref{tq12}, \ref{tq8} follows from Theorem \ref{chi2} and versions of Theorems \ref{tq3} and \ref{tq4}  using the same reasoning as in  subsection \ref{subsec9.4}.
Proof of Theorem \ref{tchi5} is akin to proof of Theorem \ref{tq5} and is omitted.
\subsection{Proof of Theorems of section \ref{sec7} and Theorem \ref{th82} \label{subsec9.7}}
    Lemma \ref{lc1} given below  allows to carry over corresponding reasoning for Brownian bridge $b(t)$, $t \in (0,1)$, instead of empirical distribution functions.
 \begin{lemma}\label{lc1} For any $x > 0$, we have
\begin{equation}\label{lm1}
{P}_{F_n}(nT^2(\hat F_n - F_0) < x) -  {P}\,(T^2(b(t) + \sqrt{n}(F_n(t) - F_0(t))) < x)  = o(1)
 \end{equation}
 uniformly onto sequences c.d.f.'s $F_n$ such that $T(F_n - F_0)  < cn^{-1/2}$.
 \end{lemma}
 If $\sqrt{n}(F_n - F_0) \to G$  in Kolmogorov - Smirnov distance, (\ref{lm1}) has been proved Chibisov \cite{chib} without any statements of uniform convergence.

 Lemma \ref{lc1} follows from Lemmas \ref{lc2} and \ref{lc4} given below after implementation of Hungary construction (see Th. 3, Ch. 12,  section 1, \cite{wel}).

 \begin{lemma}\label{lc2} For any $x > 0$, we have
 \begin{equation}\label{lm2}
 \begin{split}&
 \mathbf{ P}\,(T^2(b(F_n(t)) + \sqrt{n}(F_n(t) - F_0(t))) < x) \\& -  {\mathbf P}\,(T^2(b(t) + \sqrt{n}(F_n(t) - F_0(t))) < x)  = o(1)
 \end{split}
 \end{equation}
 uniformly onto sequences of c.d.f.'s $F_n$ such that $T(F_n - F_0)  < cn^{-1/2}$.
 \end{lemma}
 Lemma \ref{lc2} follows from Lemmas \ref{lc3} and \ref{lc4} given below.
 \begin{lemma}\label{lc3} There holds
 \begin{equation}\label{lm100}
 {\mathbf E}\,[|T^2(b(F_n(t)))- T^2(b(t))|]  < c T^{1/4}(F_n - F_0).
 \end{equation}
\end{lemma}
\begin{proof} We have
\begin{equation}\label{lm3}
\begin{split}&
 \mathbf{E}^2\, [\,|T^2(b(F_n(t))- T^2(b(t))|] \le \mathbf{E}^2\,[|(T(b(F_n(t))- T(b(t)))\,(T(b(F_n(t)) + T(b(t)))|]\\&
\le  \mathbf{E}\,[((T(b(F_n(t)))- T(b(t)))^2]\,\mathbf{E}\, [(T(b(F_n(t))) + T(b(t)))^2]\\& \le
C\,  \mathbf{E}\, [((T(b(F_n(t))- T(b(t)))^2]  \le C \mathbf{E}\, [T^2(b(F_n(t)) -b(t)))] \\& =
C \,\int_0^1 (F_n(t) - F^2_n(t) - 2\min (F_n(t), F_0(t)) + 2F_n(t)F_0(t) + F_0(t) - F_0^2(t) \, dt\\& =C \int_0^1 F_n(t) + F_0(t) -2 \min(F_n(t), F_0(t)) - (F_n(t) - F_0(t))^2 \,dt\\&
= C \,\int_0^1 |F_n(t) - F_0(t)| - (F_n(t) - F_0(t))^2 \,dt\\& \le
C\, \int_0^1\, |F_n(t) - F_0(t)| \, dt\, \le\, T^{1/2}(F_n - F_0).
\end{split}
\end{equation} \end{proof}
\begin{lemma}\label{lc4}
Densities of c.d.f.'s $\mathbf{P}\,(T^2(b(t) + n^{1/2}(F_n(t) - F_0(t))) \le x)$ are uniformly bounded onto the set of all c.d.f. $F_n$ such that $nT^2(F_n -F_0) <C$. Here $C$ is arbitrary.
\end{lemma}
\begin{proof} Brownian bridge $b(t)$  admits representation
\begin{equation*}
b(t) = \sum_{j=1}^\infty \frac{\xi_j}{\pi j} \psi_j(t)
\end{equation*}
where $\psi_j(t) = \sqrt{2}\,\sin(\pi j t)$ and $\xi_j$, $1 \le j < \infty$, are i.i.d. Gaussian random variables, $\mathbf{E}\, \xi_j = 0$ and $ \mathbf{E}\, \xi_j^2 = 1$.

Therefore, if $F_n(t) = \sum_{j=1}^\infty \theta_{nj} \psi_j$, then
\begin{equation}\label{lm4}
T^2(b(t) + n^{1/2}(F_n(t) - F_0(t))) = \sum_{j=1}^\infty \Bigl(\frac{\xi_j}{\pi j} + n^{1/2} \theta_{nj}\Bigr)^2.
\end{equation}
The right hand-side of (\ref{lm4}) is a sum of independent random variables. Thus it suffices to show that, for any $C$, random variables
\begin{equation*}
(\xi_1 + n^{1/2} \theta_{n1})^2 + (\xi_2/2 + n^{1/2} \theta_{n2})^2
\end{equation*}
have uniformly bounded densities  onto $n^{1/2}|\theta_{n1}| \le C$ and
 $n^{1/2}|\theta_{n2}| \le C$.

 Densities $(\xi_1 + a)^2$ and $(\xi_2 + b)^2$  have wellknown analytical form, and proof of uniform boundedness of densities of $(\xi_1 + a)^2 + \frac{1}{4}(\xi_2 + b)^2$ with $|a| \le C$ and
 $|b| \le C$ is obtained by routine technique. We omit these standard estimates. \end{proof}

For proof of  Theorem \ref{tcm} it suffices to prove {\it ii.} Hungary construction allows to reduce reasoning  to proof of corresponding statement for Brownian bridge $b(t)$, $t \in [0,1]$. Thus Theorem \ref{tcm} follows from Theorem \ref{th82}.

\begin{proof}[ Proof of Theorem  \ref{th82}] Denote $\zetab = \{\zeta_j\}_1^\infty$, $\zeta_j = \sigma_j \xi_j$.

Suppose opposite that (\ref{hy85}) does not valid. Then there is subsequence of vectors $\etab_{n} =\{\eta_{nj}\}_1^\infty \in \Upsilon(a)$ such that we have
\begin{equation} \label{lhm5}
\lim_{n \to \infty}\mathbf{P}\,(T(\etab_n + \zetab) \le x_\alpha) \ge  1 - \alpha.
\end{equation}
Denote $\theta_{nj} = \kappa_j \eta_{nj}$, $1 \le j < \infty$.

There are $\thetab = \{\theta_j\}_1^\infty$ and subsequence $n_i \to  \infty$ such that
$  \theta_{n_ij} \to \theta_j$ as $i \to \infty$ for each $j$, $1 \le j < \infty$.

Therefore there are sequences  $C_k \to \infty$ and  $i_k \to \infty$ as $k \to \infty$, such that
\begin{equation} \label{lhm6}
\lim_{k \to \infty}\frac{ \sum_{j < C_k} \theta_{n_{i_k}j}^2 }{ \sum_{j < C_k} \theta_j^2 } = 1
\end{equation}
and
\begin{equation} \label{lhm206}
\lim_{k \to \infty}\, \sum_{j < C_k} (\theta_{n_{i_k}j}  -\theta_j)^2 = 0
\end{equation}

We consider two cases.
\vskip 0.15cm
{\it i.} There holds
\begin{equation*}
\lim_{k \to \infty} \, \sum_{j > C_k} \theta_{n_{i_k}j}^2  = 0.
\end{equation*}
\vskip 0.15cm
{\it ii.} There holds
\begin{equation*}
 \sum_{j > C_k} \,\theta_{n_{i_k}j}^2  > c \quad \mbox{for all} \quad k > k_0.
 \end{equation*}
 If {\it i.} holds, we have
 \begin{equation}\label{uxx1}
 \mathbf{E}\,\Bigl(\sum_{j > C_k} \kappa_j\zeta_j\,\theta_{n_{i_k}j}\, \Bigr)^2 =  \sum_{j > C_k} \sigma_j^2\, \theta_{n_{i_k}j}^2\, = o(1).
 \end{equation}
 By (\ref{lhm206}), we get
 \begin{equation}\label{uxx2}
 \mathbf{E} \,\left(\sum_{j < C_k} \kappa_j\zeta_j\,(\theta_{n_{i_k}j} -\theta_j)\right)^2 = \sum_{j < C_k} \kappa_j^2\,\sigma_j^2(\theta_{n_{i_k}j}  -\eta_j)^2 = o(1).
 \end{equation}
By (\ref{uxx1}) and (\ref{uxx2}), we get
 \begin{equation*}
 \begin{split}&
\mathbf{P}\Bigl(\,\sum_{j=1}^\infty \Bigl(\kappa_j\,\zeta_j +  \theta_{n_{i_k}j}\Bigr)^2 \,  < x_\alpha\,\Bigr)\\& = \mathbf{P}\Bigl(\,\sum_{j < C_k} \Bigl
(\kappa_j \,\zeta_j +  \theta_{n_{i_k}j}\Bigr)^2 \,  +
\sum_{j > C_k} \kappa_j^2\zeta_j^2\ < x_\alpha\,(1 + o_P(1))\,\Bigr)\\& = \mathbf{P}\Bigl(\,\sum_{j < C_k}(\kappa_j \zeta_j + \theta_j)^2  + \,\sum_{j > C_k} \kappa_j^2\zeta_j^2\,  < x_\alpha\,(1 + o_P(1))\,\Bigr)\\& <
\,\mathbf{P}\Bigl(\,\sum_{j=1}^\infty \kappa_j^2\zeta_j^2 < x_\alpha\,\Bigr)\,(1 + o(1)).
\end{split}
 \end{equation*}
where the last inequality follows from Lemma \ref{lum2} given below.
\begin{lemma}\label{lum2} Let  $\thetab = \{\theta_j\}_1^\infty$ be such that $\sum_{j=1}^\infty \theta_j^2  > c$. Then there holds
 \begin{equation}\label{lhm2}
 \mathbf{P}\,\Bigl(\,\sum_{j=1}^\infty \kappa_j^2\zeta_j^2\, < x_\alpha\Bigr) > \mathbf{P}\,\Bigl(\,\sum_{j=1}^\infty(\kappa_j \zeta_j + \theta_j)^2  < x_\alpha\Bigr).
 \end{equation}
 \end{lemma}
 \begin{proof} For simplicity of notation the reasoning will be provided for $\theta_1 \ne 0$. Implementing Anderson Theorem \cite{an}, we get
\begin{equation}\label{lhm4}
\begin{split}&
\mathbf{P}\,\Bigl(\,\sum_{j=1}^\infty \Bigl(\kappa_j\zeta_j + \theta_j\Bigr)^2 < x_\alpha \Bigr)\\& =
(2\pi)^{-1/2} \int_{-\kappa_1^{-1}\sigma_1^{-1}\sqrt{x_\alpha}- \eta_1}^{\kappa_1^{-1}\sigma_1^{-1}\sqrt{x_\alpha}- \eta_1} \exp \Bigl\{-\frac{x^2}{2}\Bigr\} \mathbf{P}\Bigl(\,\sum_{j=2}^\infty \Bigl(\kappa_j\zeta_j + \theta_j\Bigr)^2 < x_\alpha - \,(\kappa_1\sigma_1x + \theta_1)^2\Bigr)\, d\,x\\& \le
(2\pi)^{-1/2} \int_{-\kappa_1^{-1}\sigma_1^{-1}\sqrt{x_\alpha}- \eta_1}^{\kappa_1^{-1}\sigma_1^{-1}\sqrt{x_\alpha}- \eta_1} \exp\Bigl\{-\frac{x^2}{2}\Bigr\} \mathbf{P}\Bigl(\,\sum_{j=2}^\infty \kappa_j^2\zeta_j^2 < x_\alpha - \,(\kappa_1\sigma_1x + \theta_1)^2\Bigr)\, d\,x\\& =
\mathbf{P}\,\Bigl(\, (\kappa_1\zeta_1 + \theta_1)^2 + \,\sum_{j=2}^\infty \kappa_j^2 \zeta_j^2 \ < x_\alpha\, \Bigr) <  \mathbf{P}\,\Bigl(\,\sum_{j=1}^\infty \kappa_j^2\zeta_j^2\,  < x_\alpha\Bigr).
\end{split}
\end{equation}
For the proof of last inequality in (\ref{lhm4}) it suffices to note that $\mathbf{P}( \kappa_1\zeta_1^2 < x) > \mathbf{P}( (\kappa_1\zeta_1 + \theta_1)^2 < x)$ for $x \in (0, x_\alpha)$, and, for any $\delta$, $0 < \delta < x_\alpha$, there is $\delta_1 > 0$ such that the function $\mathbf{P}( \kappa_1\zeta_1^2 < x) - \mathbf{P}( (\kappa_1 \zeta_1 + \theta_1)^2 < x)- \delta_1$ is positive onto interval $(\delta, x_\alpha)$.\end{proof}
Suppose {\it ii.} holds. We suppose $n_{i_k} = n$. This allows to implement more simple notation. Then we have
\begin{equation}\label{lum5}
T(\etab_n + \zetab) =  = \sum_{j< C_n} (\kappa_j\zeta_j + \theta_{nj})^2 +
J_{2n},
\end{equation}
where
\begin{equation}\label{lum6}
\begin{split}&
J_{2n}  = \sum_{j \ge C_n }\,\kappa_j^2 \zeta_j^2 + 2\sum_{j \ge C_n } \, \kappa_j\zeta_j\theta_{nj}\\& + \sum_{j \ge C_n}\,\theta_{nj}^2 = J_{21n} + 2 J_{22n}  +J_{23n}.
\end{split}
\end{equation}
We have
\begin{equation}\label{lum7}
J_{21n} = o_P(1) \quad \mbox{and}\quad J_{22n} \le J_{21n}^{1/2}\, J_{23n}^{1/2} = o_P(1).
\end{equation}
By (\ref{lum5}) - (\ref{lum7}), implementing Anderson Theorem \cite{an}, we get that, for any $0< \delta< c/2$, there holds
\begin{equation}\label{lum9}
\begin{split}&
\mathbf{P}\Bigl(\sum_{j=1}^\infty \Bigl(\kappa_j\zeta_j +  \theta_{nj}\Bigr)^2 <\, x\,\Bigr)
\le \mathbf{P}\Bigl(\sum_{j< C_n}  \Bigl(\kappa_j\zeta_j  + \theta_{nj}\,\Bigr)^2  \le x - c - o_P(1)\,\Bigr)\\&
\le \mathbf{P}\Bigl(\sum_{j< C_n}\,\kappa_j^2\zeta_j^2  \le x - c + \delta\Bigr)(1+ o(1))
\\&\le \mathbf{P}\Bigl(\sum_{j=1}^\infty \kappa_j^2 \zeta_j^2  \le x - c + 2\delta\Bigr)(1+ o(1)) \, < \mathbf{P}\Bigl(\sum_{j=1}^\infty \kappa_j^2 \zeta_j^2  \le x\Bigr),
\end{split}
\end{equation}
where last inequality follows from Proposition 7.1 in \cite{lif}.
 \end{proof}
\begin{proof}[ Proof of version of Theorem  \ref{tq3}]
Let (\ref{con2}) hold. Then we have
\begin{equation*}
n\sum_{j=1}^\infty\frac{\theta_{nj}^2}{\pi^2 j^2} \ge n\sum_{j< c_2 k_n}\frac{\theta_{nj}^2}{\pi^2 j^2} \ge c_2^{-2} n k_n^{-2}\sum_{j< c_2 k_n}\theta_{nj}^2 \asymp 1.
\end{equation*}
By (\ref{cru}), this implies sufficiency. \end{proof}

\begin{proof}[ Proof of version of Theorem  \ref{tq4}] Let (\ref{con3}) hold. Then we have
\begin{equation}\label{om202}
\begin{split}&
n\sum_{j=1}^\infty\frac{\theta_{nj}^2}{\pi^2 j^2} = n\sum_{j<c_2k_n}\frac{\theta_{nj}^2}{\pi^2 j^2} + n\sum_{j>c_2k_n}\frac{\theta_{nj}^2}{\pi^2 j^2}\\& \le  o(1) + (c_2 k_n)^{-2}
n \sum_{j > c_2 k_n} \theta_{nj}^2 \asymp o(1) +  (c_2k_n)^{-2} n^{1-2r} = O(c_2^{-2}).
\end{split}
\end{equation}
Since $c_2$  is arbitrary,  then, by (\ref{cru}), (\ref{om202}) implies sufficiency. \end{proof}

\begin{proof}[ Proof of  Theorem \ref{tom1}] Proof of {\it i} akin to  proof of {\sl i.} in Theorem  \ref{tq1}. The statement follows  from (\ref{con2}) and Lemma \ref{lom25} provided below.

\begin{lemma} \label{lom25} Let $f_n \in \mathbb{B}^s_{2\infty}(c_1)$  and $cn^{-r}\le \|f_n\| \le Cn^{-r}
$. Then, for $k_n = C_1 n^{(1-2r)/2}(1  +o(1))$  with $C_1^{2s} > 2c_1/c$, there holds
\begin{equation*}
\sum_{j=1}^{k_n} \theta_{nj}^2 > \frac{c}{2} n^{-2r}.
\end{equation*}
\end{lemma}
 Proof of Lemma \ref{lom25} is akin to proof of Lemma \ref{ld1} and is omitted.

 Reasoning in  proof of {\it ii.} is akin to  proof of {\sl ii.} in Theorem  \ref{tq1}. Suppose  opposite. Then there are $f = \sum_{j=1}^\infty \tau_{j}\,\phi_j  \notin \mathbb{B}^s_{2\infty}$  and a sequence $m_l, m_l \to \infty$ as $l \to \infty$, such that (\ref{u5}) holds. Define sequences $\etab_l$, $n_l$ and $\tilde f_l$ by the same way as in the proof of Theorem \ref{tq1}.

Then we have
\begin{equation*}
n_l \asymp C_l^{-1/(2r)} m_l^{s/r} = C_l^{-1/(2r)} m_l^{\frac{2}{1 - 2r}}.
\end{equation*}
Therefore we get
\begin{equation*}
m_l \asymp C_l^{(1-2r)/(4r)}n_l^{\frac{1-2r}{2}}.
\end{equation*}
Hence we get
\begin{equation}\label{omu12}
n_l \sum_{j=1}^{\infty} \frac{\eta_{lj}^2}{j^2} \le n_l m_l^{-2} \sum_{j=m_l}^{\infty} \eta_{lj}^2 \asymp n_l^{1-2r}m_l^{-2} \asymp C_l^{\frac{2r-1}{2r}}=   o(1) .
\end{equation}
By Theorem \ref{tcm},  (\ref{omu12}) implies  inconsistency of   sequence of alternatives $\tilde f_l$. \end{proof}

\begin{proof}[ Proof of Theorem \ref{tom6}] By Lemma \ref{lc1}, it suffices to prove that, for any $\varepsilon > 0$, there is $n_0(\varepsilon)$ such that, for $n > n_0(\varepsilon)$, the following inequality holds
\begin{equation}\label{pl7}
\begin{split}&
|\mathbf{P}( T^2(b(F_n(t)+F_{1n}(t) - F_0(t)) +  \sqrt{n}(F_n(t) + F_{1n}(t) - 2 F_0(t))) > x_\alpha)\\& - \mathbf{P}( T^2(b(F_{n}(t)) +  \sqrt{n}(F_{n}(t) - F_0(t))) > x_\alpha)| < \varepsilon.
\end{split}
\end{equation}
Since $T$ is a norm, by Lemma \ref{lc4},  proof of (\ref{pl7}) is reduced to  proof that, for any $\delta_1 > 0$, there hold
\begin{equation}\label{pl9}
\mathbf{P}(|T(b(F_n(t)+F_{1n}(t) - F_0(t))) - T(b(F_{n}(t)))| > \delta_1) = o(1),
\end{equation}
and there is sequence $\delta_n$, $\delta_n \to 0$ as  $n \to \infty$, such that there holds
\begin{equation}\label{pl11}
n^{1/2} |T(F_n(t)+F_{1n}(t) -2 F_0(t)) - T(F_{n}(t)- F_0(t))|  <\delta_n.
\end{equation}
Note that
\begin{equation}\label{pl13}
\begin{split}&
|T(b(F_n(t)+F_{1n}(t) - F_0(t))) - T(b(F_{n}(t)))|\\& \le T(b(F_n(t))+F_{1n}(t) - F_0(t)) - b(F_{n}(t)))
\end{split}
\end{equation}
and
\begin{equation}\label{pl14}
 |T(F_n(t)+F_{1n}(t) -2 F_0(t)) - T(F_{n}(t)- F_0(t))| \le T (F_{1n}(t)- F_0(t)).
\end{equation}
By Lemma \ref{lc2}, we have
\begin{equation}\label{pl15}
\mathbf{E}\, T^2(b(F_n(t)+F_{1n}(t) - F_0(t)) - b(F_{n}(t))) \le T^{1/4}(F_{1n} - F_0) =o(1).
\end{equation}
By (\ref{pl13}) and (\ref{pl15}), we get (\ref{pl9}).

Since sequence of alternatives $f_{1n}$ is inconsistent, we have
\begin{equation}\label{pl16}
 nT^2(F_{1n}(t)- F_0(t)) = o(1)
 \end{equation}
as $n \to \infty$. By (\ref{pl14}) and (\ref{pl16}),  we get (\ref{pl11}). \end{proof}

Theorem \ref{tcm}, G1  and  B reduce  proof of Theorem \ref{tom4}  to the analysis of sums $\sum_{ck_n < j < Ck_n} \theta_{nj}^2$ with $C > c$. Such an analysis  has been provided in details in subsection \ref{subsec9.4} with another parameters $r$ and $s$. We omit  proof of Theorem \ref{tom4}.


\begin{thebibliography}{99}

\bibitem{an} Anderson, T. (1955)  The integral of  a symmetric unimodal function. { \it Proc.Amer.Math.Soc.} {\bf 6}(1) 170-176.

\bibitem{au}  Autin, F., Clausel,M., Jean-Marc Freyermuth, J. and  Marteau  C. (2018). Maxiset point of view for signal detection in inverse problems. arxiv 1803.05875.


\bibitem{chib} Chibisov, D.M. (1965)  An investigation of the asymptotic power of tests of fit. {\it Theor.Prob. Appl.} {\bf 10} 421 -437.

\bibitem{co} Cohen, A., DeVore, R., Kerkyacharian, G. and Picard, D. (2001). Maximal spaces with given rate
of convergence for thresholding algorithms, {\it Appl. Comput. Harmon. Anal.} {\bf 11} 167 – 191



 \bibitem{en} Engl, H., Hanke, M. and Neubauer, A. (1996). Regularization of Inverse Problems. Kluwer Academic
Publishers.

\bibitem{er90} Ermakov, M.S. (1990) Minimax detection of a signal in a Gaussian white noise.
{\it Theory Probab. Appl.}, {\bf 35} 667-679.

 \bibitem{er97}   Ermakov, M.S. (1997). Asymptotic minimaxity of chi-squared tests. {\it Theory Probab. Appl.} {\bf 42} 589–-610.

  \bibitem{er03}   Ermakov, M.S. (2003). On asymptotic minimaxity of kernel-based tests. {\it  ESAIM Probab. Stat.} {\bf 7} 279–-312

\bibitem{er04} Ermakov, M.S.  (2006). Minimax detection of a signal in the  heteroscedastic Gaussian
white noise.
 {\it J. Math. Sci. (NY)},
{\bf 137}  4516-4524.

\bibitem{er15} Ermakov, M.S. (2017). On consistent hypothesis testing.  {\it J. Math. Sci. (NY)},
{\bf 225}  751-769.

 \bibitem{er18} Ermakov, M.S. (2018). On asymptotically minimax nonparametric detection of signal in Gaussian white noise. {\it Zapiski Nauchnih Seminarov POMI RAS.} {\bf 474} 124-138 (in Russian), arxiv.org 1705.07408.

     \bibitem{dal} Gine E. and Nickl R. (2015) Mathematical Foundation of Infinite--Dimensional Statistical Models. Cambridge University Press Cambridge

         \bibitem{gr}  Gretton A.,  Borgwardt K.,  Rasch M.,  Sch¨olkopf B. and  Smola A. A kernel two-sample test.
Journal of Machine Learning Research, 13(Mar):723–773, 2012.

    \bibitem{ih} Ibragimov,I.A. and Khasminskii, R.Z. (1977). On the estimation of infinitely dimensional parameter in Gaussian white noise. {\it Dokl.AN USSR} {\bf 236} 1053-1055.

\bibitem{ing87} Ingster, Yu.I. (1987). On comparison of the minimax properties of Kolmogorov, $\omega^2$ and $\chi^2$-tests. {\it Theory. Probab. Appl.} {\bf 32} 346-350.

\bibitem{ing02} Ingster,Yu.I.  and Suslina,I.A. (2002). {\it Nonparametric Goodness-of-fit Testing under Gaussian Models.}
Lecture Notes in Statistics {\bf 169} Springer: N.Y.

\bibitem{ing12}   Ingster,Yu. I.,  Sapatinas, T. and Suslina, I. A. (2012) {\it Minimax signal detection in ill-posed inverse problems}. --- {\it Ann. Statist.}, {\bf 40}  1524 – 1549.

\bibitem{jo}   Johnstone, I. M. (2015). Gaussian estimation. Sequence and wavelet models. {\it Book Draft} http://statweb.stanford.edu/~imj/

    \bibitem{ker93} Kerkyacharian, G. and Picard, D. (1993). Density estimation by kernel and wavelets methods: optimality of Besov spaces.
{\it Statist. Probab. Lett.} {\bf 18} 327 - 336.

\bibitem{ker02} Kerkyacharian, G. and Picard, D. (2002). Minimax or maxisets? {\it  Bernoulli} {\bf 8}, 219- 253.

\bibitem{la} Laurent, B., Loubes, J. M., and Marteau, C. (2011).
Testing inverse problems: a
direct or an indirect problem?
{\it J. Statist. Plann. Inference}
{\bf 141} 1849-–1861.

\bibitem{les} Le Cam, L. and Schwartz, L. (1960). A necessary and sufficient conditions for the existence of consistent estimates. {\it Ann.Math.Statist.} {\bf 31}  140-150.

    \bibitem{le73} Le Cam, L. (1973). Convergence of estimates under dimensionality restrictions. {\it Ann.Statist.} {\bf 1} 38-53.

        \bibitem{le} Lehmann, E.L. and Romano, J.P. (2005). {\it Testing Statistical Hypothesis}. Springer Verlag, NY.

\bibitem{lep}    Lepski, O.V. and Tsybakov, A.B.(2000). Asymptotically exact nonparametric hypothesis
testing in sup-norm and at a fixed point. {\it Probab. Theory Related
Fields}, {\bf 117}:1, 17–48.

\bibitem{lif} Lifshits, M. (2012) Lectures on Gaussian Processes. Springer. NY.

\bibitem{rio}  Rivoirard, V. (2004). Maxisets for linear procedures.  {\it Statist. Probab. Lett.} {\bf 67}  267-275

\bibitem{wel} Shorack, G.R. and Wellner, J.A. (1986) Empirical Processes with Application to Statistics. J.Wiley Sons  NY

 \bibitem{sch} Schwartz, L. (1965). On Bayes procedures. {\it Z.Wahrsch.Verw. Gebiete} {\bf 4} 10-26.

\bibitem{ts} Tsybakov, A. (2009). {\it Introduction to Nonparametric Estimation.  }
Berlin: Springer.

\bibitem{ul} Ulyanov, P. L. (1964). {\it On  Haar series}. Mathematical Sbornik. {\bf 63(105)}:2 356-391. In Russian.
\end{thebibliography}
\end{document}